\documentclass{siamltex}

\usepackage{graphicx}
\usepackage{amsmath}
\usepackage{amssymb}
\usepackage{url}
\usepackage{algorithm}
\usepackage{algorithmic}

\usepackage{tikz}
\usetikzlibrary{arrows,automata,positioning,shadows}
\usepackage{caption}

\definecolor{darkorange}{rgb}{1,0.6,0}
\definecolor{oryell}{rgb}{1,0.75,0}
\definecolor{yellow}{rgb}{1,0.85,0}
\definecolor{cverde}{rgb}{0,.5,0}
\definecolor{light}{rgb}{.8,.8,.8}
\definecolor{dark}{rgb}{.5,.5,.5}

\numberwithin{equation}{section}

\newcommand{\bm}[1]{\mathbf{#1}}

\newcommand{\e}{{\mathrm{e}}}
\newcommand{\ii}{{\mathbf{i}}}

\newcommand{\R}{{\mathbb{R}}}

\newcommand{\cU}{{\mathcal{U}}}
\newcommand{\cV}{{\mathcal{V}}}

\newcommand{\cF}{{\mathcal{F}}}

\newcommand{\cC}{{\mathcal{C}}}

\newcommand{\ttt}[1]{\texttt{#1}}
\newcommand{\pib}{{\boldsymbol{\pi}}}

\DeclareMathOperator{\circu}{circ}

\usepackage{ifthen}
\usepackage{float}
\floatstyle{ruled}
\newfloat{algorithm}{htb}{alg}
\floatname{algorithm}{Algorithm}
\newcounter{algo@row}
\newcounter{algo@rowindent}
\newcommand{\algofont}[1]{\textbf{#1}}
\newcommand{\algonumbersize}[1]{\scriptsize{#1}}
\newcommand{\algopreitem}[1][\arabic{algo@row}]{\texttt{\algonumbersize{#1}}}
\newcommand{\algoitemskip}{\hspace{\value{algo@rowindent}cc}}
\newenvironment{algo}{\vskip.3em\small%
  \begin{list}{\algopreitem\texttt{\algonumbersize{:}}}{%
      \usecounter{algo@row}%
      \setcounter{algo@rowindent}{0}%
      \setlength{\itemindent}{2em}%
      \setlength{\labelwidth}{2em}
      \setlength{\parsep}{0cm}%
    }%
}{
  \end{list}\vskip-.5em
}
\newcommand{\algonewnestedopen}[2]{
  \newcommand{#1}[1][]{%
    \ifthenelse{\equal{##1}{}}{\item}{\item[{\algopreitem[##1]}]}
    \algoitemskip\algofont{#2}%
    \addtocounter{algo@rowindent}{1}%
    \ignorespaces
  }
}
\newcommand{\algonewnestedaux}[2]{
  \newcommand{#1}[1][]{
    \addtocounter{algo@rowindent}{-1}
    \ifthenelse{\equal{##1}{}}{\item}{\item[{\algopreitem[##1]}]}
    \algoitemskip\algofont{#2}%
    \addtocounter{algo@rowindent}{+1}%
    \ignorespaces
  }
}
\newcommand{\algonewnestedclose}[2]{
  \newcommand{#1}[1][]{
    \addtocounter{algo@rowindent}{-1}
    \ifthenelse{\equal{##1}{}}{\item}{\item[{\algopreitem[##1]}]}
    \algoitemskip\algofont{#2}%
    \ignorespaces
  }
}
\newcommand{\algonewcommand}[2]{
  \newcommand{#1}[1][default]{
    \ifthenelse{\equal{##1}{default}}{\item}{\item[{\algopreitem[##1]}]}%
    \algoitemskip\algofont{#2}%
    \ignorespaces
  }%
}
\newcommand{\algonewkeyword}[2]{\newcommand{#1}{\algofont{#2}}}
\algonewcommand{\STATE}{\ignorespaces}
\algonewcommand{\INPUT}{Input: }
\algonewcommand{\pINPUT}{\phantom{Input: }}
\algonewcommand{\COMPUTE}{Compute: }
\algonewcommand{\OUTPUT}{Output: }
\algonewcommand{\pOUTPUT}{\phantom{Output: }}
\algonewnestedopen{\IF}{if }
\algonewnestedaux{\ELSEIF}{else if }
\algonewnestedaux{\ELSE}{else }
\algonewnestedclose{\ENDIF}{end if }
\algonewnestedopen{\FOR}{for }
\algonewnestedclose{\ENDFOR}{end for }
\algonewnestedopen{\WHILE}{while }
\algonewnestedclose{\ENDWHILE}{end while }
\algonewcommand{\BREAK}{break}%
\algonewkeyword{\For}{for }%
\algonewkeyword{\To}{to }%
\algonewkeyword{\Do}{do }%
\algonewkeyword{\If}{if }%
\algonewkeyword{\Then}{then }%
\algonewkeyword{\Else}{else }%
\algonewkeyword{\End}{end }%
\algonewkeyword{\AND}{and }%
\algonewkeyword{\True}{true }%
\algonewkeyword{\False}{false }%
\algonewkeyword{\Call}{call }%
\algonewkeyword{\Function}{function }%

\begin{document}           

\pagestyle{myheadings}
\markboth{A. Concas, C. Fenu, G. Rodriguez, and R. Vandebril}{The seriation
problem in the presence of a multiple Fiedler value}

\title{The seriation problem in the presence of a double Fiedler value}\footnotetext{Version \today}

\author{A. Concas\thanks{Dipartimento di Matematica e Informatica,
Universit\`a di Cagliari, viale Merello 92, 09123 Cagliari, Italy. E-mail: 
\texttt{anna.concas@unica.it}, \texttt{kate.fenu@unica.it},
\texttt{rodriguez@unica.it}.
Research supported in part by INdAM-GNCS.}
\and C. Fenu\footnotemark[1]
\and G. Rodriguez\footnotemark[1]
\and R. Vandebril\thanks{Department of Computer Science, KU Leuven, Celestijnenlaan 200A,
3001 Heverlee, Leuven, Belgium. E-mail: \texttt{raf.vandebril@cs.kuleuven.be}}}

\maketitle

\begin{abstract} 

Seriation is a problem consisting of seeking the best enumeration order of a
set of units whose interrelationship is described by a bipartite graph, that
is, a graph whose nodes are partitioned in two sets and arcs only connect nodes
in different groups. An algorithm for spectral seriation based on the use of
the Fiedler vector of the Laplacian matrix associated to the problem was
developed by Atkins et al., under the assumption that the Fiedler value is
simple. In this paper, we analyze the case in which the Fiedler value of the
Laplacian is not simple, discuss its effect on the set of the admissible
solutions, and study possible approaches to actually perform the computation.
Examples and numerical experiments illustrate the effectiveness of the proposed
methods.
\end{abstract}

\section{Introduction}\label{sec:intro}

By \emph{seriation}, we refer to an important ordering problem
that aims at recovering the best enumeration order of a set of units in terms
of a chosen correlation function.
Such order can be chronological, or any sequential structure which
characterizes the data. The notion of seriation has been formulated in several
ways and appears in various fields, such as archaeology, anthropology,
psychology, and biology~\cite{brusco06,genome,matharcheo,mirkin1984}. The first
systematic formalization of the seriation problem was made by Petrie in
1899~\cite{petrie}, even if the term seriation was used before in archaeology;
see Concas et al.~\cite{pqser19} for an overview.

When the ordering is chronological, seriation concerns relative dating of
objects or events, which is employed when absolute dating methods cannot be
used. This means that the order lacks a direction, in the sense that the units
are placed in a sequence which can be read in both directions. Seriation finds
another application in \emph{de novo} genome sequencing. In this case, from a
randomly oversampled DNA strand (the so-called \emph{reads}) the whole sequence
is reconstructed. Oversampling is necessary to increase the probability of all
parts being covered. The reads which overlap are then considered as similar and
their ordering is obtained by placing similar reads close to each other.

In all the applications, seriation data are usually given in terms of a matrix
of size $n\times m$, called the \emph{data matrix}, whose row and/or column
indices represent the elements to be ordered. In archaeology, the rows of the
data matrix correspond to the units (e.g., the sites) and the columns represent
the types of the archaeological findings detected in the units. Each unit is
characterized by the presence of certain artefacts, which are in turn
classified in types. Piana Agostinetti and Sommacal~\cite{ps05}, the authors
refer to the data matrix as either incidence matrix or abundance matrix,
depending on the archaeological data representation. In the first case, the
data are reported by using a binary representation, i.e., an element in the
position $(i,j)$ is equal to $1$ if type $j$ is present in the unit $i$, and
$0$ otherwise. In the second case, each element of the data matrix reports the
number of objects belonging to a certain type in a given unit, or its
percentage. In this paper, following the usual terminology used in complex
networks theory, we will refer to the binary representation as an
\emph{adjacency matrix}. More details can be found in~\cite{pqser19}. The
purpose of determining a relative chronology consists of obtaining an ordering
of the rows and columns of the data matrix that places the nonzero entries
close to its diagonal. Given the variety of applications, some software
packages have been developed in the past to manipulate seriation data;
see~\cite{pqser19} for an overview.

A spectral algorithm for the solution of the seriation problem was considered
by Atkins et al.~\cite{atkins1998spectral}, and an optimized Matlab
implementation has recently been proposed by Concas et al.~\cite{pqser19}.
Each solution is a permutation of the nodes which solves a particular
optimization problem.
The method is based on the use of the Fiedler vector of the Laplacian
matrix associated with the problem, and describes the set of solutions in terms
of a data structure known as a PQ-tree.
In this paper, we discuss the implications of the presence of a multiple
Fiedler value, an issue which has been disregarded up to now. Our 
interest is mainly for the case of multiplicity two, for which we illustrate 
the effects on the set of solutions.

The plan of the paper is the following. Section \ref{sec:mathback} reviews the
necessary mathematical background, sets up the terminology to be used in the
rest of the paper, and describes the data structures
used to store the solutions of the seriation problem.
The spectral algorithm and the special case of a multiple Fiedler value are
discussed in Section~\ref{sec:multfied}.
In Section~\ref{sec:casestudies}, we extensively analyze three example networks
whose Laplacian admits a double Fiedler value, showing the consequences on the
set of solutions of the seriation problem.
Section~\ref{sec:methods} describes two practical algorithms for computing the
admissible solutions, and Section~\ref{sec:numexp} reports some numerical
results. Finally, Section~\ref{sec:last} contains concluding remarks.

\section{Mathematical background}\label{sec:mathback}

Here we review some mathematical
concepts that will be used in the following.
Matrices will be denoted by upper case roman letters, vectors by lower case
bold letters, and their elements by lower case doubly and singly indexed
letters, respectively.

Let $G$ be a simple graph with $n$ nodes. The adjacency matrix $F\in\R^{n\times n}$ associated to $G$ contains in position $(i,j)$ the weight of the edge connecting node $i$ to node $j$. 
If the two nodes are not connected, then $f_{ij}=0$.
If a graph is unweighted, then the weights are either 0 or 1.
The adjacency matrix is symmetric if the graph is undirected.

The (unnormalized) \emph{graph Laplacian} of a symmetric irreducible matrix 
$F\in\R^{n\times n}$ is the symmetric, positive semidefinite matrix 
$$
L=D-F,
$$
where $D=\diag(d_1,\ldots,d_n)$ is the \emph{degree matrix}, whose $i$th
diagonal element equals the sum of the weights of all the edges starting from
node $i$ in the undirected network defined by $F$, that is, 
$d_i=\sum_{j=1}^n f_{ij}$.
In the case of an unweighted graph, $d_i$ is simply the number of nodes
connected to node $i$. It is immediate to observe that $0$ is an eigenvalue of
the graph Laplacian, with associated eigenvector $\bm{e} =
(1,\dots,1)^T\in\R^n$, and that all the eigenvalues
$\lambda_1=0\leq\lambda_2\leq\dots\leq\lambda_n$ are non-negative. 

The smallest eigenvalue of $L$ with associated eigenvector orthogonal to
$\bm{e}$ is called the \emph{Fiedler value}, or the \emph{algebraic
connectivity}, of the graph described by $F$.
The corresponding normalized eigenvector is the \emph{Fiedler
vector}~\cite{fiedler1973algebraic,fiedler1975property,fiedler1989laplacian}.
Alternatively, the Fiedler value may be defined to be any vector $\bm{x}$ that achieves the minimum
$$
\min_{\bm{x}^{T}\bm{e}=0,\ \bm{x}^{T}\bm{x}=1}\bm{x}^{T}L\bm{x}.
$$

In this paper we describe the seriation problem in terms of bipartite graphs,
since the interrelationship between the units to be reordered can be expressed
in terms of such graphs. A \emph{bipartite graph} $G$ is a graph whose vertices
can be divided into two disjoint sets $U$ and $V$ such
that every edge connects a node in $U$ to one in $V$. In our archaeological
metaphor the sets $U$ and $V$, containing $n$ and $m$ nodes respectively,
represent the units and the types of the findings. Hence, the adjacency data
matrix $A\in\R^{n\times m}$ associated to the seriation problem can be
interpreted as the matrix which describes the connections in the bipartite
graph associated to the problem and it is obtained by setting $a_{i,j}=1$ if
unit $i$ contains objects of type $j$ and $0$ otherwise.

One approach for solving the seriation problem is based on the construction of a
symmetric \emph{similarity matrix}
$S$, whose elements $s_{ij}$ describe the likeness of the nodes $i,j\in U$~\cite{brainerd1951place,robinson1951method}.
One possible definition for it is through the product $S=AA^T$, being $A$ the
adjacency matrix of the bipartite graph associated to the problem.
In this case, $s_{ij}$ equals the number of types shared between unit $i$ and
unit $j$. 
The largest value on each row is the diagonal element, which reports the
number of types associated to each unit.
By applying the same permutation to the rows and columns of $S$ in order to
cluster the largest values close to the main diagonal, one obtains the
permutation of the rows of $A$ that brings close the units more similar for
what concerns types.
It is worth noting that this rows and columns permutation is
not uniquely defined.

The \emph{Robinson method}~\cite{robinson1951method} is a technique
based on a different similarity matrix. Starting from an abundance matrix $A \in \R^{n \times m}$ whose
entries are in percentage form (the sum of each row is 100), it computes the
similarity matrix $S$ by a particular rule, leading to a symmetric matrix of
order $n$ with entries between $0$ (rows with no types in common) and $200$,
which corresponds to units containing exactly the same types.
Then, the method searches for a permutation matrix $P$ such that $PSP^T$ has
its largest entries as close as possible to the main diagonal. The same
permutation determines the chronological order for the units.

The procedure of finding a permutation matrix $P$ is not uniquely specified.
One way to deal with it is given by the so called \emph{Robinson's form}, which
places larger values close to the main diagonal, and lets off-diagonal entries
be nonincreasingly ordered moving away from the main diagonal.
Such a matrix is also called $R$-matrix, or it is said to be in $R$-form;
see~\cite{pqser19} for details.
A symmetric matrix is pre-$R$ if and only if there exists a simultaneous
permutation of its rows and columns which takes it to Robinson's form, so
it corresponds to a well-posed ordering problem;
see~\cite{chepoi1997,laurent2017lex,laurent2017,prea2014,seston2008}.

A subset of the possible permutations of the elements of a set can be encoded
in a data structure called \emph{PQ-tree}, originally introduced by Booth and
Lueker~\cite{booth1976testing}.
A PQ-tree $T$ over a set $U = \{u_1,u_2,\dots,u_n\}$ is a rooted tree whose
leaves are elements of $U$ and whose internal (non-leaf) nodes are
distinguished as either P-nodes or Q-nodes. 
The only difference between them is the way in which their children are
treated. In particular, the children of a P-node can be arbitrarily permuted,
while the order of those of a Q-node can only be reversed. The root of the tree
can either be a P or a Q-node; see~\cite{pqser19} for a Matlab implementation
of PQ-trees.

We now briefly review the spectral algorithm for the seriation
problem introduced in~\cite{atkins1998spectral} and implemented
in~\cite{pqser19}.
Starting from a pre-R matrix, it constructs a PQ-tree describing the set of all
the row and column permutations that lead to an $R$-matrix.

Given the set of units $U=\{u_1,u_2,\dots,u_n\}$, the notation 
$i\preccurlyeq j$ indicates that $u_i$ precedes $u_j$ in a chosen ordering.
Then, a symmetric bivariate \emph{correlation function} $f$ can be used to
describe the desire for units $i$ and $j$ to be close to each other in the
sought sequence; see~\cite{atkins1998spectral}.
The aim of the algorithm is to find all index permutation vectors
$\pib=(\pi_1,\ldots,\pi_n)^T$ such that 
\begin{equation}\label{fperm}
\pi_i\preccurlyeq\pi_j\preccurlyeq\pi_k \quad \iff \quad 
f(\pi_i,\pi_j)\geq f(\pi_i,\pi_k) \quad \text{and} \quad
f(\pi_j,\pi_k)\geq f(\pi_i,\pi_k).
\end{equation}
Setting $f_{ij}=f(i,j)$ defines a matrix $F$ with the same role as the
similarity matrix $S$ aforementioned.

If a seriation data set is described by an adjacency (or abundance) matrix $A$,
we set $F=AA^T$.
If $F$ is pre-$R$, there exists a row/column permutation that takes it in
$R$-form. 
Unfortunately, this property cannot be stated in advance in general.

The approach adopted in~\cite{atkins1998spectral} (see
also~\cite{estrada2010network}) is to consider the constrained optimization
problem
$$
\begin{aligned}
&\text{minimize} & &
h(\bm{x}) = \frac{1}{2}\sum_{i,j=1}^{n} f_{ij}(x_i-x_j)^2, \\
&\text{subject to} & & \sum_i x_i = 0 \quad \text{and} \quad \sum_i x_i^2 = 1.
\end{aligned}
$$
The value of the function $h(\bm{x})$ is small for a vector $\bm{x}$ such that
each pair $(u_i,u_j)$ of highly correlated units is associated to components
$x_i$ and $x_j$ with close values.
Once the minimizing vector $\bm{x}_{\min}$ is computed, it is sorted according
to either nonincreasing or nondecreasing values, yielding
$\bm{x}_\pib=(x_{\pi_1},\ldots,x_{\pi_n})^T$.
The permutation $\pib$ of the units realizes \eqref{fperm}.

Letting $D = \diag(d_i)$ be the degree matrix, 
the previous minimization problem can be rewritten as 
\begin{eqnarray*}
\min_{\|\mathbf{x}\|=1,\ \bm{x}^T\bm{e} = 0}
\mathbf{x}^TL\mathbf{x},
\end{eqnarray*}
where $L=D-F$ is the Laplacian of the correlation matrix $F$.
The two constraints require that $\mathbf{x}$ be a unit vector orthogonal to
$\mathbf{e}$.
This shows, by the Courant--Fischer--Weyl min-max principle, that any Fiedler
vector is a solution to the constrained minimization problem.

The problem is well posed only when $F$ is pre-R.
Nevertheless, a real data set may be inconsistent, in the sense that it may
not necessarily lead to a pre-R similarity matrix.
In such cases, it may be useful to construct an approximate solution to the
seriation problem, and sorting the entries of the Fiedler vector generates an
ordering that tries to bring highly correlated elements close to each other.
We refer to such orderings as \emph{admissible permutations}.

\section{Multiple Fiedler values in seriation}\label{sec:multfied}

In this section we analyze the case of the presence of a multiple Fiedler value
and its effect on the spectral algorithm discussed above.

Let us assume that the Fiedler value has multiplicity $k$, and let
$\bm{q}_1,\ldots,\bm{q}_k$ be an orthonormal basis of the corresponding
eigenspace $\mathcal{F}$.
For each $\bm{x}\in\mathcal{F}$, there is a vector
$\tilde{\bm{y}}=(y_1,\ldots,y_k)^T$ such that
\begin{equation}\label{Fiedvec}
\bm{x}=Q_k\tilde{\bm{y}},
\end{equation}
where $Q_k=[\bm{q}_1,\ldots,\bm{q}_k]$.
We remind the reader that a solution to the seriation problem is determined by
sorting the vector $\bm{x}$ either nonincreasingly or nondecreasingly.

When $k=1$ there is in general only one permutation which solves the problem,
together with its reverse. There are multiple solutions if the eigenvector
$\bm{x}$ has $\ell$ multiple equal components. In this case, there will be
$\ell!$ solutions.

When $k>1$, after extending $Q_k$ to a square orthogonal matrix $Q$, we can
write $\bm{x}=Q\bm{y}$, with
$$
\bm{y} = \begin{bmatrix} \tilde{\bm{y}} \\ \bm{0} \end{bmatrix}\in\R^n.
$$
Although it is clear that only the first $k$ entries are relevant in
determining $\bm{x}$, it is not trivial to understand how many permutations are
allowed to sort $\bm{x}$ when the components of $\tilde{\bm{y}}$ vary.

Let us analyze the situation where $\bm{q}_i=\bm{e}_i$, the vectors of the
canonical basis in $\R^n$, $i=1,\ldots,k$, so that we may set $Q=I$,
Even in the case $\bm{x}=\bm{y}$, the conclusion is not trivial.
If the first $k$ components of $\bm{y}$ are different from zero and distinct,
then the indexes associated to the last $n-k$ zero components admit 
$(n-k)!$ equivalent permutations.
We can consider such indexes in the whole vector $\bm{y}$ as grouped in
a unique ``vector'' index, as the corresponding components all share the same
position in each possible sorting. 
Under this assumption, the number of different orderings for $\bm{y}$
is $(k+1)!$.
Substituting to the vector index all its possible permutations, the 
number of admissible solutions grows to 
\begin{equation}\label{wrong}
(k+1)!(n-k)!. 
\end{equation}
If there are groups of equal components in $\tilde{\bm{y}}$, this number is
going to increase accordingly.
The truth is that in the general case, that is when $Q\neq I$, the number of
admissible permutations depends upon the structure of the Fiedler vectors.

Concas et al.~\cite{pqser19} pointed out that non pre-R matrices can lead to
Laplacian matrices whose Fiedler value is not simple and conjectured, through
the following simple example, that the number of permutations
\eqref{wrong} may be incorrect. 

Let us consider the seriation problem described by the bipartite graph depicted
in Figure~\ref{fig:cycle} (left).
The relationship between nodes on the left (units) and nodes on the right
(types) is represented by edges. The adjacency matrix of this graph and the
resulting similarity matrix are, respectively
\[
E = \begin{bmatrix}
1 & 1 & 0 & 0 & 0\\
0 & 1 & 1 & 0 & 0\\
0 & 0 & 1 & 1 & 0\\
0 & 0 & 0 & 1 & 1\\
1 & 0 & 0 & 0 & 1
\end{bmatrix} \qquad \text{and} \qquad
S = EE^T = \begin{bmatrix}
2 & 1 & 0 & 0 & 1\\
1 & 2 & 1 & 0 & 0\\
0 & 1 & 2 & 1 & 0\\
0 & 0 & 1 & 2 & 1\\
1 & 0 & 0 & 1 & 2
\end{bmatrix}.
\]
Note that $S$ can be seen as the adjacency matrix of the graph shown
in Figure~\ref{fig:cycle} (right).

A solution to the seriation problem does not exist in this case, since the
associated graph describes
a \emph{cycle}: each unit is similar to surrounding units and the two extremal
units are similar to each other.
This leads to a non pre-R similarity matrix.
As shown in~\cite{pqser19}, the Fiedler value of the Laplacian $L = D-F$ has
multiplicity 2, so each vector belonging to the \emph{Fiedler plane} can be
sorted to obtain the admissible permutations of the units. 
In the same paper, the authors considered a randomized approximated approach,
which will be discussed in Section~\ref{sec:methods}, to determine such
permutations.
They found only 5 admissible permutations, much less than the number
$(k+1)!(n-k)!=(2+1)!(5-2)!=36$ determined in \eqref{wrong}.

In this paper, we will show that this estimate for the number of admissible
permutations was wrong, nevertheless, we will confirm the fact that when a
Fiedler value is multiple some constraints are imposed on the admissible
permutations of the units. In particular, we will show that their number
does not only depend on the multiplicity of the Fielder value, but also on the
structure of the underlying bipartite graph.

In the following, we often focus on the number of permutations found. Referring
to such a number is significant only to show that, in the cases analyzed, the
number of admissible solutions is always smaller than the forecast given by
\eqref{wrong}.
We stress the fact that solving the seriation problem consists of listing all
the admissible permutations of the nodes. Any theoretical analysis or
numerical algorithm must be able to produce such result.

\section{Three case studies}\label{sec:casestudies}

In this section, to gain insight in the behavior of other similar examples, we
analyze three different graphs whose Laplacian exhibits a double Fiedler
value: the modified star graph, the cycle graph, and the generalized Petersen
graph.

\subsection{The modified star graph}\label{sec:stargraph}

Consider the bipartite graph represented in Figure~\ref{fig:star} (left) whose
associated data matrix is
\begin{equation}
\label{eq:incidModStar}
E = \left[ \begin{array}{c}
\bm{e}_{n-1}^T \\ \hline I_{n-1}
\end{array} \right] \in  \R^{n \times (n-1)},
\end{equation}
where $\bm{e}_k=(1,\ldots,1)^T\in\R^k$, and $I_k$ denotes the identity matrix
of size $k$.
As already stated, $e_{i,j}=1$ indicates that unit $i$ contains  objects of
type $j$.

\begin{figure}
\begin{minipage}{.49\textwidth}
\begin{center}
\begin{tikzpicture}[-,>=stealth',auto,node distance=1cm,
	on grid,semithick, every state/.style={ball color=oryell,shading=ball,draw=none,text=black,inner sep=0pt}]
	\scriptsize
\node[state,ball color=yellow!70,shading=ball](A){$1$};
\node[state,ball color=yellow!70,shading=ball](C)[below=of A]{$2$};
\node[state,ball color=yellow!70,shading=ball](D)[below=of C]{$3$};
\node[state,ball color=yellow!70,shading=ball](E)[below=of D]{$4$};
\node[state,ball color=yellow!70,shading=ball](B)[below=of E]{$5$};
\node[state,ball color=yellow!70,shading=ball](F)[below=of B]{$6$};
\node[state](M)[right=of A,xshift=4cm]{$1$};
\node[state](H)[below=of M]{$2$};
\node[state](I)[below=of H]{$3$};
\node[state](L)[below=of I]{$4$};
\node[state](G)[below=of L]{$5$};
\path (A) edge (M);
\path (A) edge (H);
\path (A) edge (I);
\path (A) edge (L);
\path (A) edge (G);
\path (C) edge (M);
\path (D) edge (H);
\path (E) edge (I);
\path (B) edge (L);
\path (F) edge (G);
\end{tikzpicture}
\end{center}
\end{minipage}
\begin{minipage}{.49\textwidth}
\begin{center}
\begin{tikzpicture}[-,>=stealth',auto,node distance=2cm,
	on grid,semithick, every state/.style={ball color=red,shading=ball,draw=none,text=black,inner sep=0pt}]
	\scriptsize
\node[state,ball color=blue!30,shading=ball](A){$2$};
\node[state,ball color=blue!30,shading=ball](C)[below=of A, xshift=2cm, yshift=0.4cm]{$3$};
\node[state,ball color=blue!30,shading=ball](D)[below=of C, xshift=-0.8cm]{$4$};
\node[state,ball color=blue!30,shading=ball](B)[below=of A, xshift=-2cm, yshift=0.4cm]{$6$};
\node[state,ball color=blue!30,shading=ball](E)[below=of B, xshift=0.8cm]{$5$};
\node[state,ball color=blue!30,shading=ball](F)[below=of A]{$1$};
\path (F) edge (C);
\path (F) edge (B);
\path (F) edge (D);
\path (F) edge (E);
\path (F) edge (A);
\end{tikzpicture}
\end{center}
\end{minipage}
\caption{The bipartite graph associated with the data matrix $E$ \eqref{eq:incidModStar} with $n=6$ (left) which leads to the star graph $\mathcal{S}_6$ (right).}
\label{fig:star}
\end{figure}
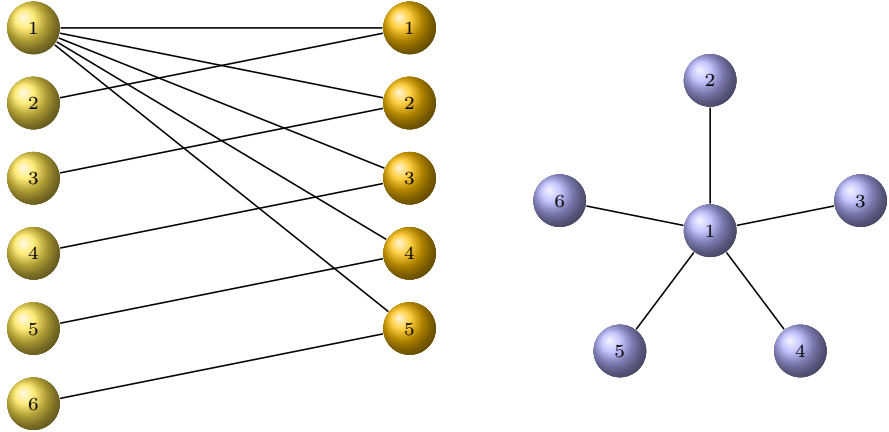

The resulting similarity and Laplacian matrices are given by
\begin{equation}\label{laplstar}
S = EE^T = \left[ \begin{array}{c|c}
n-1 & \bm{e}_{n-1}^T \\
\hline
\bm{e}_{n-1} & I_{n-1}
\end{array} \right], \qquad
L = D-S = \left[ \begin{array}{c|c}
n-1 & -\bm{e}_{n-1}^T \\
\hline
-\bm{e}_{n-1} & I_{n-1}
\end{array} \right], 
\end{equation}
where $D = \diag(d_1, \dots, d_n)$, $d_i = \sum_{j=1}^n s_{ij}$,
is the degree matrix associated to $S$.

The matrix $S$ can be interpreted as the adjacency matrix of a star graph; see
Figure~\ref{fig:star} (right).
A star graph $\mathcal{S}_n$ is a connected graph with $n$ vertices and $n-1$
edges, where one vertex, the \emph{center} of the star, has degree $n-1$ and the other $n-1$ vertices have
degree $1$. It is a special case of a complete bipartite graph in which one set
has one vertex and the other set contains the remaining $n-1$ vertices. 

Both the Laplacian and the similarity matrix \eqref{laplstar} are arrowhead
matrices, that is, real symmetric matrices of the form
\begin{equation}\label{arrowhead}
\begin{bmatrix}
\alpha & \mathbf{z}^T \\
\mathbf{z} & \Delta \\
\end{bmatrix}
\end{equation}
where $\alpha$ is a scalar, $\mathbf{z}\in\R^{n-1}$, and
$\Delta=\diag(\delta_1,\dots,\delta_{n-1})$.
From the Cauchy interlacing theorem \cite{wilkinson} for the eigenvalues of
Hermitian matrices, it follows that the sorted eigenvalues $\lambda_1, \dots,
\lambda_n$ of~\eqref{arrowhead} interlace the sorted elements $\delta_i$ of the diagonal matrix
$\Delta$.
If $\delta_1\geq \delta_2\geq \dots \geq \delta_{n-1}$ and if the eigenvalues
$\lambda_i$, $i=1,\dots,n$, are sorted accordingly, then the following
inequality holds
\begin{equation}\label{interlacing}
\lambda_1\geq \delta_1\geq \lambda_2 \geq \delta_2 \geq \dots \geq
\lambda_{n-1} \geq \delta_{n-1} \geq \lambda_n.
\end{equation}
If $\delta_i=\delta_{i-1}$ for some $i$, the above inequality implies that
$\delta_i$ is an eigenvalue of the arrowhead matrix~\eqref{arrowhead}
considered.

The following theorem identifies the eigenvalues of the Laplacian matrix in the
case of a star graph.

\begin{theorem}\label{theo:star}
Let $S$ be the adjacency matrix of a star graph $\mathcal{S}_n$. Then, the
spectrum of the Laplacian matrix $L$ consists of the three eigenvalues
0, 1, and $n$, with the second having multiplicity $n-2$.
\end{theorem}

\begin{proof}
A well known result states that the smallest eigenvalue of the Laplacian  is
$\lambda_n=0$.
From the Cauchy interlacing theorem applied to the matrix $L$ in~\eqref{laplstar}, it follows (see
\eqref{interlacing}) that 1 is an eigenvalue with multiplicity $n-2$.
Setting $\mathbf{v}=(-(n-1),1,\dots,1)^T\in\R^n$, we see that
$L\mathbf{v} = n\mathbf{v}$, so that $\lambda_1=n$.
\end{proof}
\smallskip

\begin{corollary}
Let $S$ be an adjacency matrix of a star graph. Then, the Fiedler value has
multiplicity $n-2$ and the $n-2$ Fiedler vectors have a null component in the
position corresponding to the central node index.
\end{corollary}

\begin{proof}
Without loss of generality we can assume that the first node is the central one
of degree $n-1$. 
To determine the Fiedler vectors one has to solve the homogeneous linear system
$(L-I_n)\bm{v}=0$, whose coefficient matrix is of the form
$$
L - I_n = \begin{bmatrix}
n-2 & -\mathbf{e}_{n-1}^T \\
-\mathbf{e}_{n-1} & 0 \\
\end{bmatrix}.
$$
The last $n-1$ equations of the system show that the first component of the
Fiedler vectors is always $0$, while the first equation implies that
the sum of their components is $0$.
\end{proof}
\smallskip

Since we are focusing on the case of a double Fiedler value, let us consider 
the \emph{modified star graph}.
In the bipartite graph of Figure~\ref{fig:star}, we add
$n-4$ nodes to the set of the types, and connect each of these nodes to two
consecutive nodes in the set of units, except the first ones. We obtain
the bipartite graph in Figure~\ref{fig:star_mod} (left).
The seriation data matrix associated to this graph is
\begin{equation}
\label{eq:mod_star}
E = \left[ \begin{array}{c|c}
\bm{e}_{n-1}^T & \bm{0}_{n-4}^T \\ 
\hline 
I_{n-1} & B_{n-1,3}
\end{array} \right] \in  \R^{n \times (2n-5)},
\end{equation}
where $\bm{0}_k\in\R^k$ is a null vector, and
$B_{k,\ell}\in\R^{k\times(k-\ell)}$ is the lower bidiagonal matrix whose
elements are 1 on the main diagonal and on the sub-diagonal, and zero
otherwise.

The resulting similarity matrix is
$$
S = \left[ \begin{array}{c|c}
n-1 & \bm{e}_{n-1}^T \\
\hline
\bm{e}_{n-1} & \begin{array}{c|c} T_{n-3} & O \\ \hline O & I_2 \end{array}
\end{array} \right], 
$$
where $O$ denotes a null matrix of suitable size and $T_{n-3}$ is the
tridiagonal matrix
\begin{equation}
\begin{bmatrix}\label{trid}
2 & 1 \\
1 & 3 & 1 \\
& \ddots & \ddots & \ddots \\
& & 1 & 3 & 1 \\
& & & 1 & 2 
\end{bmatrix}.
\end{equation}
The similarity matrix $S$ can be seen as the adjacency matrix of the 
modified star graph in
Figure~\ref{fig:star_mod} (right), which we denote by $\widehat{\mathcal{S}}_6$.

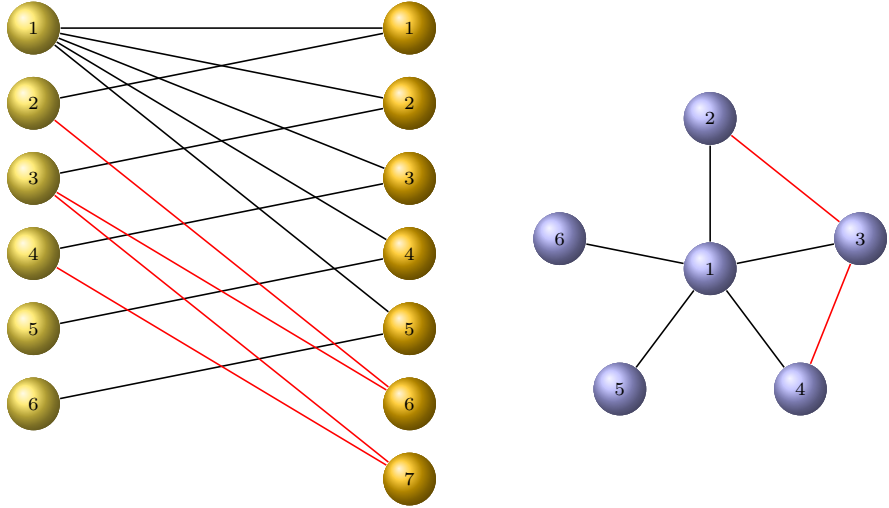
\begin{figure}[ht!]
\begin{minipage}{.49\textwidth}
\begin{center}
\begin{tikzpicture}[-,>=stealth',auto,node distance=1cm,
	on grid,semithick, every state/.style={ball color=oryell,shading=ball,draw=none,text=black,inner sep=0pt}]
	\scriptsize
\node[state,ball color=yellow!70,shading=ball](A){$1$};
\node[state,ball color=yellow!70,shading=ball](C)[below=of A]{$2$};
\node[state,ball color=yellow!70,shading=ball](D)[below=of C]{$3$};
\node[state,ball color=yellow!70,shading=ball](E)[below=of D]{$4$};
\node[state,ball color=yellow!70,shading=ball](B)[below=of E]{$5$};
\node[state,ball color=yellow!70,shading=ball](F)[below=of B]{$6$};
\node[state](M)[right=of A,xshift=4cm]{$1$};
\node[state](H)[below=of M]{$2$};
\node[state](I)[below=of H]{$3$};
\node[state](L)[below=of I]{$4$};
\node[state](G)[below=of L]{$5$};
\node[state](N)[below=of G]{$6$};
\node[state](O)[below=of N]{$7$};
\path (A) edge (M);
\path (A) edge (H);
\path (A) edge (I);
\path (A) edge (L);
\path (A) edge (G);
\path (C) edge (M);
\path (D) edge (H);
\path (E) edge (I);
\path (B) edge (L);
\path (F) edge (G);
\path[red] (C) edge (N);
\path[red] (D) edge (N);
\path[red] (D) edge (O);
\path[red] (E) edge (O);
\end{tikzpicture}
\end{center}
\end{minipage}
\begin{minipage}{.49\textwidth}
\begin{center}
\begin{tikzpicture}[-,>=stealth',auto,node distance=2cm,
	on grid,semithick, every state/.style={ball color=red,shading=ball,draw=none,text=black,inner sep=0pt}]
	\scriptsize
\node[state,ball color=blue!30,shading=ball](A){$2$};
\node[state,ball color=blue!30,shading=ball](C)[below=of A, xshift=2cm, yshift=0.4cm]{$3$};
\node[state,ball color=blue!30,shading=ball](D)[below=of C, xshift=-0.8cm]{$4$};
\node[state,ball color=blue!30,shading=ball](B)[below=of A, xshift=-2cm, yshift=0.4cm]{$6$};
\node[state,ball color=blue!30,shading=ball](E)[below=of B, xshift=0.8cm]{$5$};
\node[state,ball color=blue!30,shading=ball](F)[below=of A]{$1$};
\path (F) edge (C);
\path (F) edge (B);
\path (F) edge (D);
\path (F) edge (E);
\path (F) edge (A);
\path[red] (A) edge (C);
\path[red] (C) edge (D);
\end{tikzpicture}
\end{center}
\end{minipage}
\caption{Bipartite graph represented by matrix~\eqref{eq:mod_star} (left) and resulting graph $\widehat{\mathcal{S}}_6$ (right). The edges in red are the added ones.}
\label{fig:star_mod}
\end{figure}

The Laplacian matrix of $S$ is given by 
\begin{equation}\label{laplstar1}
L = D-S = \left[ \begin{array}{c|c}
n-1 & -\bm{e}_{n-1}^T \\
\hline
-\bm{e}_{n-1} & \begin{array}{c|c} \widetilde{T}_{n-3} & O^{\strut} \\ \hline O & I_2 
	\end{array}
\end{array} \right],
\end{equation}
where $O$ denotes a null matrix of suitable size and $\widetilde{T}_{n-3}$ is
like \eqref{trid}, but with the elements in the
sub- and in the super-diagonal of opposite sign.

The following theorem explains the behavior of the Fiedler value of the
Laplacian matrix in the case of the modified star graph
$\widehat{\mathcal{S}}_n$.

\begin{theorem}\label{theo:star2}
Let $S$ be the adjacency matrix of a modified star graph $\widehat{\mathcal{S}}_n$. Then,
the spectrum of the Laplacian matrix $L$ \eqref{laplstar1} contains the three
eigenvalues 0, 1, and $n$, with the second having multiplicity $2$, while the
remaining $n-4$ eigenvalues are in the interval $(1,5)$.
\end{theorem}

\begin{proof}
A direct computation shows that $\lambda_1=0$, $\lambda_2=\lambda_3=1$, and
$\lambda_n=n$, are eigenvalues of $L$ with associated eigenvectors
$$
\bm{v}_1=\bm{e}_n, \qquad
\bm{v}_2=\begin{bmatrix} 0 \\ -\bm{e}_{n-2} \\ n-2 \end{bmatrix}, \qquad
\bm{v}_3=\begin{bmatrix} 0 \\ -\bm{e}_{n-3} \\ n-3 \\ 0 \end{bmatrix}, \qquad
\bm{v}_n=\begin{bmatrix} 1-n \\ \bm{e}_{n-1} \end{bmatrix},
$$
where $\bm{e}_k=(1,\ldots,1)^T\in\R^k$.

By a simple application of the Gram-Schmidt process, we see that any vector
orthogonal to $\bm{v}_1$, $\bm{v}_2$, and $\bm{v}_n$ has a null first and
last component, like $\bm{v}_3$.
So, the remaining $n-4$ eigenvectors take the form 
$$
\bm{v}_i=\begin{bmatrix} 0 \\ \tilde{\bm{v}} \\ 0 \end{bmatrix}, \qquad
i=4,\ldots,n-1,
$$
with $\tilde{\bm{v}}\in\R^{n-2}$.
Given the expression \eqref{laplstar1} of matrix $L$, any such vector
$\tilde{\bm{v}}$ is an eigenvector of the principal submatrix
$$
\widetilde{L} = \left[ 
\begin{array}{c|c} \widetilde{T}_{n-3} & \bm{0}_{n-3} \\
\hline \bm{0}_{n-3}^T & 1^{\strut} \end{array} \right].
$$
Besides the eigenvalue $\lambda_2=1$, the remaining eigenvalues of
$\widetilde{L}$ are those of $\widetilde{T}_{n-3}$.

The Gershgorin circle theorems applied to $\widetilde{T}_{n-3}$ yields
$1\leq\lambda_i<5$, $i=3,\ldots,n-1$.
It is immediate to observe that $\lambda_3=1$ with associated eigenvector
$\bm{e}_{n-3}$.
It is a simple eigenvalue because a symmetric tridiagonal matrix with nonzero
subdiagonal elements has distinct eigenvalues \cite{ortega1960}. This completes
the proof.
\end{proof}

\begin{corollary}\label{theo:star2bis}
Let $S$ be the adjacency matrix of a modified star graph
$\widehat{\mathcal{S}}_n$. Then, its Fiedler value is equal to $1$ and has
multiplicity $2$.
\end{corollary}

In the case of the modified star graph $\widehat{\mathcal{S}}_n$, an orthogonal basis for the eigenspace $\cF$ corresponding to the Fiedler value is given by
$$
Q_2 = \begin{bmatrix} \mathbf{q_1} & \mathbf{q_2} \end{bmatrix}, 
$$
where $\mathbf{q_1}=\mathbf{v}_3$ and $\mathbf{q_2}=\mathbf{v}_2$.
For the sake of simplicity, we do not normalize the two eigenvectors.
Letting $\widetilde{\mathbf{y}}=(\alpha,\beta)^T\in\R^2\setminus\{(0,0)\}$,
every $\mathbf{x}\in \cF$ can be expressed as
\begin{equation}\label{eigvecStar}
\mathbf{x} = Q_2\widetilde{\mathbf{y}} 
= \begin{bmatrix} 
0 & 0  \\ 
-1 & -1 \\ 
\vdots & \vdots \\
-1 & -1 \\
n-3 & -1 \\
0 & n-2 
\end{bmatrix} \begin{bmatrix} \alpha \\ \beta \end{bmatrix} 
= \begin{bmatrix} 0 \\ -\alpha-\beta \\ \vdots \\ -\alpha-\beta \\
(n-3)\alpha-\beta \\ (n-2)\beta \end{bmatrix}.
\end{equation} 
The admissible permutations are then related to the possible reorderings of the
entries of $\mathbf{x}\in\cF$, and these sortings depend on the values of the
coefficients $\alpha$ and $\beta$.
We remark that they cannot be both zero, as $\bm{x}$ is an eigenvector.

We let $x_1=0$, $x_2=-\alpha-\beta$, $x_{n-1}=(n-3)\alpha-\beta$, and
$x_n=(n-2)\beta$.
The relative position of such components is governed by the following
inequalities, where we initially consider only strict inequality
\begin{equation}\label{ineqs}
\begin{cases}
x_2>x_1, \quad & \text{for } \alpha<-\beta, \\
x_{n-1}>x_1, \quad & \text{for } \alpha>\frac{1}{n-3}\beta, \\
x_n>x_1, \quad & \text{for } \beta>0, \\
x_{n-1}>x_2, \quad & \text{for } \alpha>0, \\
x_n>x_2, \quad & \text{for } \alpha>-(n-1)\beta, \\
x_n>x_{n-1}, \quad & \text{for } \alpha<\frac{n-1}{n-3}\beta.
\end{cases}
\end{equation}
When considering a particular ordering of the vector $\bm{x}$, multiple index
permutations are produced by permuting the components of the vector
$\bm{x}_2=(x_2,\ldots,x_2)^T\in\R^{n-3}$, containing the equal components in
\eqref{eigvecStar}.
To identify such permutations we consider the following cases:
\begin{enumerate}

\item\label{item1} $\alpha,\beta > 0$:
in correspondence to the three inequalities
\begin{equation}\label{starcase1}
0 < \alpha<\frac{1}{n-3}\beta, \qquad
\frac{1}{n-3}\beta<\alpha<\frac{n-1}{n-3}\beta, \qquad
\alpha>\frac{n-1}{n-3}\beta,
\end{equation}
we find the following increasingly ordered vectors $\bm{x}$, 
\begin{equation}\label{ab0}
\begin{bmatrix} \bm{x}_2 \\ x_{n-1} \\ x_1 \\ x_n \end{bmatrix}, \qquad
\begin{bmatrix} \bm{x}_2 \\ x_1 \\ x_{n-1} \\ x_n \end{bmatrix}, \qquad
\begin{bmatrix} \bm{x}_2 \\ x_1 \\ x_n \\ x_{n-1} \end{bmatrix},
\end{equation}
respectively.
In this case, we obtain $(n-3)!$ index permutations for each of the three
vectors, that is, $3(n-3)!$ admissible permutations.
They result from permuting the elements of $\bm{x}_2$.

For example, for $n=5$ we obtain the 6 permutations contained in the columns of
the following matrix
\begin{equation*}
\begin{bmatrix}
2 & 3 & 2 & 3 & 2 & 3 \\
3 & 2 & 3 & 2 & 3 & 2 \\
4 & 4 & 1 & 1 & 1 & 1 \\
1 & 1 & 4 & 4 & 5 & 5 \\
5 & 5 & 5 & 5 & 4 & 4
\end{bmatrix}.
\end{equation*}

\item\label{item2} $\alpha>0>\beta$:
now, the three inequalities
\begin{equation}\label{starcase2}
0<\alpha<-\beta, \qquad
-\beta<\alpha<-(n-1)\beta, \qquad
\alpha>-(n-1)\beta,
\end{equation}
correspond to the sorted vectors
\begin{equation}\label{a0b}
\begin{bmatrix} x_n \\ x_1 \\ \bm{x}_2 \\ x_{n-1} \end{bmatrix}, \qquad
\begin{bmatrix} x_n \\ \bm{x}_2 \\ x_1 \\ x_{n-1} \end{bmatrix}, \qquad
\begin{bmatrix} \bm{x}_2 \\ x_n \\ x_1 \\ x_{n-1} \end{bmatrix},
\end{equation}
which originate $3(n-3)!$ more possible index permutations for $\bm{x}$.

For $n=5$, we obtain 
\begin{equation*}
\begin{bmatrix}
5 & 5 & 5 & 5 & 2 & 3 \\
1 & 1 & 2 & 3 & 3 & 2 \\
2 & 3 & 3 & 2 & 5 & 5 \\
3 & 2 & 1 & 1 & 1 & 1 \\
4 & 4 & 4 & 4 & 4 & 4
\end{bmatrix}.
\end{equation*}

\end{enumerate}

The above cases are exhaustive. Indeed, the inequalities $\alpha,\beta<0$ and
$\alpha<0<\beta$ produce permutations which are the reverse of the ones
already considered in \ref{item1} and \ref{item2}, respectively.
The total number of permutations accounted for so far is
$$
N_1 = 6(n-3)!.
$$

We now consider equalities in \eqref{ineqs}, that is, we seek the values of the
parameters $\alpha$ and $\beta$ for which some components of the vector
$\bm{x}$ in \eqref{eigvecStar} become equal, besides those of $\bm{x}_2$.

It is important to remark that if two scalar components are equal, no new
permutations are introduced. 
For example, $(n-3)\alpha=\beta$ makes $x_1=x_{n-1}$, but the vector
orderings deriving from the permutation of these two components have already
been considered in the first two vectors of \eqref{ab0}.

On the contrary, when $x_2$ is equal to any of the three other different
components, then new index permutations are generated by permuting the
considered component with the entries of the vector $\bm{x}_2$.
When $\alpha\geq 0>\beta$, the special cases where $x_2=x_1$, $x_2=x_{n-1}$,
and $x_2=x_n$, correspond to the conditions
$$
\alpha=-\beta, \qquad
\alpha=0, \qquad 
\alpha=-(n-1)\beta,
$$
respectively, and lead to the sorted vectors
\begin{equation}\label{sortvecs}
\begin{bmatrix} x_n \\ \widetilde{\bm{x}}_{2,1} \\ x_{n-1} \end{bmatrix}, \qquad
\begin{bmatrix} x_n \\ x_1 \\ \widetilde{\bm{x}}_{2,n-1} \end{bmatrix}, \qquad
\begin{bmatrix} \widetilde{\bm{x}}_{2,n} \\ x_1 \\ x_{n-1} \end{bmatrix},
\end{equation}
where
$$
\widetilde{\bm{x}}_{2,k} = \begin{bmatrix} \bm{x}_2 \\ x_k \end{bmatrix}
= (x_2,\ldots,x_2,x_k)^T \in \R^{n-2}, \qquad k=1,n-1,n.
$$
Each vector in \eqref{sortvecs} produces $(n-2)!$ index permutations, from
which one must subtract those already considered in \eqref{ab0} and
\eqref{a0b}. 
For example, for the first vector of \eqref{sortvecs} the permutations
$$
\begin{bmatrix} x_n \\ \bm{x}_2 \\ x_1 \\ x_{n-1} \end{bmatrix}, \quad
\begin{bmatrix} x_n \\ x_1 \\ \bm{x}_2 \\ x_{n-1} \end{bmatrix}, 
$$
have already been accounted for in the first two vectors of \eqref{a0b}.
This leads to
$$
N_2 = 3\bigl((n-2)!-2(n-3)!\bigr) = 3(n-4)(n-3)!
$$
permutations. 
For $n=5$ we obtain 
\begin{equation*}
\begin{bmatrix}
5 & 5 & 5 & 5 & 2 & 3 \\
2 & 3 & 1 & 1 & 5 & 5 \\
1 & 1 & 2 & 3 & 3 & 2 \\
3 & 2 & 4 & 4 & 1 & 1 \\
4 & 4 & 3 & 2 & 4 & 4
\end{bmatrix}.
\end{equation*}
To conclude with, the vector $\bm{x}$ defined in \eqref{eigvecStar} possesses
\begin{equation}\label{right_star}
N = N_1+N_2 = 3(n-2)!
\end{equation}
admissible permutations for $\alpha,\beta\in\R\setminus\{(0,0)\}$.
Such permutations are one half of those foreseen by formula \eqref{wrong}, that
is, $3!(n-2)!$, confirming the conjecture that the structure of the problem
introduces some constraints on the number of admissible solutions for the
seriation problem.

\subsection{The cycle graph}\label{subsec:cycle}

The second example of a graph whose Laplacian has a multiple Fiedler value is
the cycle or circular graph $\cC_n$, whose vertices are connected in a closed
chain. The number of edges in $\cC_n$ equals the number of vertices and, since
every node has exactly two edges incident to it, every vertex has degree 2.
Hence a cycle is a regular graph, i.e., a graph in which each vertex has the
same degree $k$.

\begin{figure}
\begin{minipage}{.49\textwidth}
\begin{center}
\begin{tikzpicture}[-,>=stealth',auto,node distance=1cm,
	on grid,semithick, every state/.style={ball color=oryell,shading=ball,draw=none,text=black,inner sep=0pt}]
	\scriptsize
\node[state,ball color=yellow!70,shading=ball](A){$1$};
\node[state,ball color=yellow!70,shading=ball](C)[below=of A]{$2$};
\node[state,ball color=yellow!70,shading=ball](D)[below=of C]{$3$};
\node[state,ball color=yellow!70,shading=ball](E)[below=of D]{$4$};
\node[state,ball color=yellow!70,shading=ball](B)[below=of E]{$5$};
\node[state](F)[right=of A,xshift=4cm]{$1$};
\node[state](H)[below=of F]{$2$};
\node[state](I)[below=of H]{$3$};
\node[state](L)[below=of I]{$4$};
\node[state](G)[below=of L]{$5$};
\path (A) edge (F);
\path (A) edge (H);
\path (C) edge (H);
\path (C) edge (I);
\path (D) edge (I);
\path (D) edge (L);
\path (E) edge (L);
\path (E) edge (G);
\path (B) edge (G);
\path (B) edge (F);
\end{tikzpicture}
\end{center}
\end{minipage}
\begin{minipage}{.49\textwidth}
\begin{center}
\begin{tikzpicture}[-,>=stealth',auto,node distance=2cm,
	on grid,semithick, every state/.style={ball color=red,shading=ball,draw=none,text=black,inner sep=0pt}]
	\scriptsize
\node[state,ball color=blue!30,shading=ball](A){$1$};
\node[state,ball color=blue!30,shading=ball](C)[below=of A, xshift=2cm, yshift=0.4cm]{$2$};
\node[state,ball color=blue!30,shading=ball](D)[below=of C, xshift=-0.8cm]{$3$};
\node[state,ball color=blue!30,shading=ball](B)[below=of A, xshift=-2cm, yshift=0.4cm]{$5$};
\node[state,ball color=blue!30,shading=ball](E)[below=of B, xshift=0.8cm]{$4$};
\path (A) edge[bend left] (C);
\path (A) edge[bend right] (B);
\path (C) edge[bend left] (D);
\path (D) edge[bend left] (E);
\path (E) edge[bend left] (B);
\end{tikzpicture}
\end{center}
\end{minipage}
\caption{The bipartite graph associated with the data matrix $E$ in
\eqref{eq:incidCycle} for $n=5$ (left), which leads to the cycle graph $\cC_5$
(right).}
\label{fig:cycle}
\end{figure}
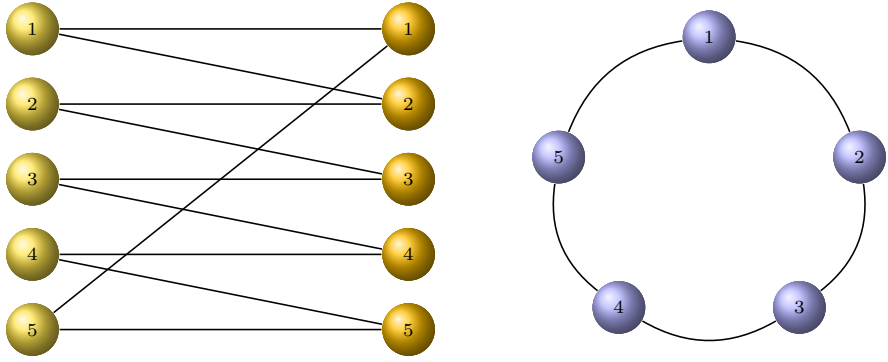

Consider the bipartite graph represented in Figure~\ref{fig:cycle} (left) with
associated data matrix
\begin{equation}
E = \left[ \begin{array}{cc}
B_{n,1}^T \\ 
\hline 
\bm{b}_n^{T\strut}
\end{array} \right] \in \R^{n \times n}.
\label{eq:incidCycle}
\end{equation}
where $B_{n,1}\in\R^{n\times(n-1)}$ is the lower bidiagonal matrix defined in
\eqref{eq:mod_star} and $\bm{b}_n=(1,\bm{0}_{n-2}^T,1)^T$, being
$\bm{0}_{k}$ the null vector of length $k$.
As $B_{n,1}^T\bm{b}_n=\bm{b}_{n-1}^{T}$,
its similarity matrix and Laplacian are, respectively,
\begin{equation}\label{sim:cycle}
S =EE^T = \left[ \begin{array}{c|c}
C_{n-1} & \bm{b}_{n-1} \\ 
\hline 
\bm{b}_{n-1}^T & 2^{\strut}
\end{array} \right],
\qquad
L = D-S = \left[ \begin{array}{c|c}
\widetilde{C}_{n-1} & -\bm{b}_{n-1} \\ 
\hline 
-\bm{b}_{n-1}^T & 2^{\strut}
\end{array} \right]
,
\end{equation}
where 
\[
C_{n-1} = \begin{bmatrix}
2 & 1 \\
1 & 2 & 1 \\
& \ddots & \ddots & \ddots \\
& & 1 & 2 & 1 \\
& & & 1 & 2 
\end{bmatrix}\in\R^{(n-1)\times (n-1)},
\]
and $\widetilde{C}_{n-1}$ is the tridiagonal matrix like $C_{n-1}$, with the
elements in the sub- and super-diagonal of opposite sign.
The matrix $S$ can be seen as the adjacency matrix of a cycle graph $C_n$; see
Figure~\ref{fig:cycle}.

The matrix $L$ is circulant, that is, it is fully specified by its first
column, while the other columns are cyclic permutations of the first one
with an offset equal to the column index~\cite{davis1979circulant}.
A basic property of a circulant matrix $C$ is that its spectrum is analytically
known. It is given by
\begin{equation}\label{spectr}
\sigma(C)=\{\widehat{C}(1), \widehat{C}(\omega),\dots,
\widehat{C}(\omega^{n-1})\},
\end{equation}
where
\begin{equation}\label{lapl}
\widehat{C}(\zeta) = \sum_{k=0}^{n-1} c_k \zeta^{-k}
\end{equation}
is the discrete Fourier transform of the first column 
$(c_0,c_1,\ldots,c_{n-1})^T$ of $C$, $\omega = \e^{\frac{2\pi\ii}{n}}$ is
the minimal phase $n$th root of unity, and $\ii$ the imaginary unit.

The next theorem states the behavior of the eigenvalues of the Laplacian matrix
in the special case of a circular graph.

\begin{theorem}\label{theo:cycle}
Let $E$ be the similarity matrix of a cycle graph with at least $n \geq 3$
vertices. Then, the eigenvalues of the Laplacian matrix $L = D-E$ are
coupled as follows
$$
\lambda_j = \lambda_{n-j+2}, \qquad 
j=2,\ldots,\left\lfloor\frac{n}{2}\right\rfloor+1,
$$ 
where $\lfloor m\rfloor$ denotes the minimal integer part of $m$.
In particular, if $n$ is odd $\lambda_1 = 0$ is the only simple eigenvalue. 
If $n$ is even, the eigenvalues $\lambda_1 = 0$ and $\lambda_{n/2}$, of
smallest and largest modulus, respectively, are the only simple ones. 
\end{theorem}

\smallskip
The property trivially results from $L$ being a symmetric circulant matrix.
For the sake of clarity, we give a simple proof.

\begin{proof}
First, we recover a well known result in graph theory which states that the
eigenvalue of smallest modulus of the Laplacian is $\lambda_1 = 0$.
Indeed, from~\eqref{spectr} and~\eqref{lapl}, it follows that the discrete
Fourier transform of the first column of $L$ is
$$
\widehat{L}(\zeta) = 2 -\zeta ^{-1}-\zeta ^{-(n-1)},
$$
and that $\lambda_1=\widehat{L}(1)=0$.
Next, let $k = 1, \dots, n-1$. From~\eqref{spectr} and~\eqref{lapl} we obtain
$$
\lambda_{k+1} = \widehat{L}(\omega^{k}) 
= 2 - e^{-\frac{2\pi i}{n}k}-e^{\frac{2\pi i}{n}k} 
= 2 - 2\cos(\theta_{k}),
$$
where $\theta_{k}=-\frac{2\pi i}{n}k$.
The thesis follows from the property $\omega^{k} = \overline{\omega^{n-k}}$.
\end{proof}

\smallskip
The theorem immediately implies the following.
\begin{corollary}
Let a graph satisfy the assumptions of Theorem~\ref{theo:cycle}. 
Then, its Fiedler value has multiplicity $2$.
\end{corollary}

\smallskip
The normalized eigenvectors of an $n\times n$ circulant matrix are the columns
of the normalized Fourier matrix, that is, 
\begin{equation}\label{eigenvcirc}
\bm{v}_j=\frac{1}{\sqrt{n}} \left(1, \omega^{(j-1)}, \omega^{2(j-1)}, \ldots,
\omega^{(n-1)(j-1)} \right)^T, \qquad j = 1, \dots, n.
\end{equation}
A basis for the eigenspace corresponding to the Fiedler value is given by
$\{\bm{v}_2,\bm{v}_n\}$, where
the entries of $\bm{v}_n$ are the conjugates of those of $\bm{v}_2$.
To obtain eigenvectors with real entries we consider the vectors
\begin{equation}\label{eig_C}
\bm{w}_1 = \frac{(\bm{v}_2+\bm{v}_n)}{2},
\qquad
\bm{w}_2 = \frac{(\bm{v}_2-\bm{v}_n)}{2\ii},
\end{equation}
with components 
$$
(\bm{w}_1)_j = \cos\frac{2(j-1)\pi}{n}, \qquad
(\bm{w}_2)_j = \sin\frac{2(j-1)\pi}{n}, \qquad
j=1,\ldots,n.
$$
These vectors are, in fact, connected to the discrete cosine transform
(DCT) and the discrete sine transform (DST), respectively. 
They have many symmetries,
$$
(\bm{w}_1)_j = (\bm{w}_1)_{n-j+2}, \qquad 
(\bm{w}_2)_j = -(\bm{w}_2)_{n-j+2}, \qquad 
j=2,\ldots,\left\lfloor\frac{n}{2}\right\rfloor+1,
$$
and more relations are valid for $n$ either odd or even.

Every Fiedler vector $\bm{x}$ lies in the eigenspace generated by $\bm{w}_1$ and
$\bm{w}_2$, so that it can be expressed as
\begin{equation}
\label{eq:cycleFv}
\bm{x} = \alpha\bm{w}_1 + \beta\bm{w}_2,
\end{equation}
for $\alpha$ and $\beta\in \R$. 
Anyway, because of the many symmetries in the vectors $\bm{w}_1$ and
$\bm{w}_2$, it is impracticable to find a general rule to find the number of
admissible permutations, i.e., of all the possible reorderings of the
components of $\bm{x}$ for any $n$. 
The task is made harder by the fact that for specific values of the
coefficients $\alpha$ and $\beta$, groups of components of the Fiedler vector
$\bm{x}$ take the same value, generating bunches of admissible permutations.
We analyzed in detail the situation for $n=4,5,6,7$, determining 
8, 15, 30, and 49 permutations, respectively.
These results will be confirmed numerically in Section~\ref{sec:numexp}.
We report here the permutations obtained for $n=4$
$$
P_{x_{(n=4)}}= \begin{bmatrix}
2 & 2 & 3 & 3 & 3 & 3 & 4 & 4 \\
3 & 3 & 2 & 2 & 4 & 4 & 3 & 3 \\
1 & 4 & 1 & 4 & 1 & 2 & 1 & 2 \\
4 & 1 & 4 & 1 & 2 & 1 & 2 & 1
\end{bmatrix}.
$$
We remark, that according to formula \eqref{wrong} the number of admissible
solutions for $n=4,5,6,7$ should be 12, 36, 144, and 720, respectively.

\subsection{The generalized Petersen graph}\label{subs:petersen}

The generalized Petersen graph is another graph whose Fiedler value has
multiplicity $2$. It was introduced by Coxeter~\cite{coxeter1950} and it was given its
name later, in 1969, by Watkins~\cite{watkins1969}. We denote it by $GPG(n,k)$.
It has $2n$ vertices and $3n$ edges given, respectively, by
$$
\begin{aligned}
V(GPG(n,k)) &= \{u_i,v_i, 1\leq i\leq n\}, \\
E(GPG(n,k)) &= \{u_iu_{i+1},u_iv_i,v_iv_{i+k}|1\leq i\leq n\},
\end{aligned}
$$
where the subscripts are expressed as integers modulo $n$ ($n\geq 5$) and $k$
is the so called ``skip''.
Let $\cU(n,k)$ (respectively, $\cV(n,k)$) be the subgraph of $GPG(n,k)$
consisting of the vertices $\{u_i|1\leq i\leq n\}$ (respectively, $\{v_i, 1\leq
i\leq n\}$) and edges $\{u_iu_{i+1}|1\leq i\leq n\}$ (respectively,
$\{v_iv_{i+k}|1\leq i\leq n\}$). We will call $\cU(n,k)$ (respectively,
$\cV(n,k)$) the outer (respectively, inner) subgraph of $GPG(n,k)$.

The $2n\times 2n$ data matrix of the bipartite graph $GPG(n,k)$ has the block
structure
\begin{equation}\label{eq:adjGPG}
E = \begin{bmatrix}
U & I_n \\
I_n & V_k
\end{bmatrix}
\end{equation}
where $I_n$ is the $n\times n$ identity matrix, the block $U$ is the adjacency
matrix of the outer subgraph $\cU(n,k)$, it coincides with the adjacency
matrix $S$ \eqref{sim:cycle} of a cycle graph, with the diagonal elements equal
to 3.
The block $V_k$ is the adjacency matrix for the inner graph $\cV(n,k)$, whose
structure is determined by the skip $k$.
The matrices $U$ and $V_k$ are circulant. They are specified by their first
column given, respectively, by 
$$
\bm{c}=(0,1,\underbrace{0,\ldots,0}_{n-3},1)^T, \qquad 
\bm{c}^{(k)}=(\bm{0}_k,1,\underbrace{0,\ldots,0}_{n-2k-1},1,\bm{0}_{k-1})^T,
$$
where $\bm{0}_j$ denotes the null vector of length $j$, or the empty vector
when $j=0$.
We will write $U=\circu(\bm{c})$ and $V_k=\circu(\bm{c}^{(k)})$.

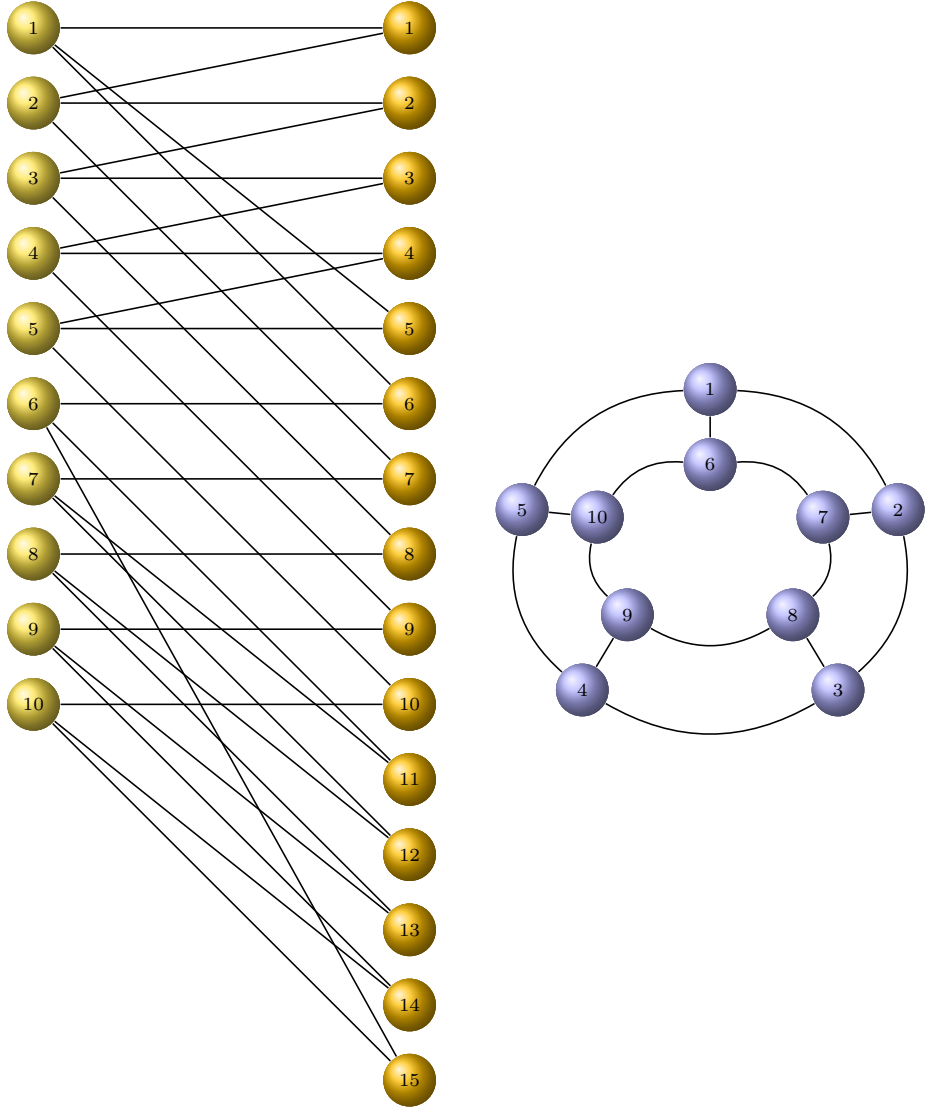
\begin{figure}[ht!]
\begin{minipage}{.49\textwidth}
\begin{center}
\begin{tikzpicture}[-,>=stealth',auto,node distance=1cm,
	on grid,semithick, every state/.style={ball color=oryell,shading=ball,draw=none,text=black,inner sep=0pt}]
	\scriptsize
\node[state,ball color=yellow!70,shading=ball](A){$1$};
\node[state,ball color=yellow!70,shading=ball](C)[below=of A]{$2$};
\node[state,ball color=yellow!70,shading=ball](D)[below=of C]{$3$};
\node[state,ball color=yellow!70,shading=ball](E)[below=of D]{$4$};
\node[state,ball color=yellow!70,shading=ball](B)[below=of E]{$5$};
\node[state,ball color=yellow!70,shading=ball](F)[below=of B]{$6$};
\node[state,ball color=yellow!70,shading=ball](G)[below=of F]{$7$};
\node[state,ball color=yellow!70,shading=ball](H)[below=of G]{$8$};
\node[state,ball color=yellow!70,shading=ball](I)[below=of H]{$9$};
\node[state,ball color=yellow!70,shading=ball](J)[below=of I]{$10$};
\node[state](L)[right=of A,xshift=4cm]{$1$};
\node[state](M)[below=of L]{$2$};
\node[state](N)[below=of M]{$3$};
\node[state](O)[below=of N]{$4$};
\node[state](P)[below=of O]{$5$};
\node[state](Q)[below=of P]{$6$};
\node[state](R)[below=of Q]{$7$};
\node[state](S)[below=of R]{$8$};
\node[state](T)[below=of S]{$9$};
\node[state](U)[below=of T]{$10$};
\node[state](V)[below=of U]{$11$};
\node[state](W)[below=of V]{$12$};
\node[state](X)[below=of W]{$13$};
\node[state](Y)[below=of X]{$14$};
\node[state](Z)[below=of Y]{$15$};
\path (A) edge (L);
\path (A) edge (P);
\path (A) edge (Q);
\path (C) edge (L);
\path (C) edge (M);
\path (C) edge (R);
\path (D) edge (M);
\path (D) edge (N);
\path (D) edge (S);
\path (E) edge (N);
\path (E) edge (O);
\path (E) edge (T);
\path (B) edge (O);
\path (B) edge (P);
\path (B) edge (U);
\path (F) edge (Q);
\path (F) edge (V);
\path (F) edge (Z);
\path (G) edge (R);
\path (G) edge (V);
\path (G) edge (W);
\path (H) edge (S);
\path (H) edge (W);
\path (H) edge (X);
\path (I) edge (T);
\path (I) edge (X);
\path (I) edge (Y);
\path (J) edge (U);
\path (J) edge (Y);
\path (J) edge (Z);
\end{tikzpicture}
\end{center}
\end{minipage}
\begin{minipage}{.49\textwidth}
\begin{center}
\begin{tikzpicture}[-,>=stealth',auto,node distance=2cm,
	on grid,semithick, every state/.style={ball color=red,shading=ball,draw=none,text=black,inner sep=0pt}]
	\scriptsize
\node[state,ball color=blue!30,shading=ball](A){$1$};
\node[state,ball color=blue!30,shading=ball](C)[below=of A, xshift=2.5cm, yshift=0.4cm]{$2$};
\node[state,ball color=blue!30,shading=ball](D)[below=of C, xshift=-0.8cm, yshift=-0.4cm]{$3$};
\node[state,ball color=blue!30,shading=ball](B)[below=of A, xshift=-2.5cm, yshift=0.4cm]{$5$};
\node[state,ball color=blue!30,shading=ball](E)[below=of B, xshift=0.8cm,yshift=-0.4cm ]{$4$};
\node[state,ball color=blue!30,shading=ball](F)[below=of A,yshift=1cm]{$6$};
\node[state,ball color=blue!30,shading=ball](G)[below=of A, xshift=1.5cm, yshift=0.3cm]{$7$};
\node[state,ball color=blue!30,shading=ball](H)[below=of C, xshift=-1.4cm, yshift=0.6cm]{$8$};
\node[state,ball color=blue!30,shading=ball](I)[below=of A, xshift=-1.5cm, yshift=0.3cm]{$10$};
\node[state,ball color=blue!30,shading=ball](L)[below=of B, xshift=1.4cm,yshift=0.6cm ]{$9$};
\path (A) edge[bend left] (C);
\path (A) edge[bend right] (B);
\path (C) edge[bend left] (D);
\path (D) edge[bend left] (E);
\path (E) edge[bend left] (B);
\path (F) edge[bend left] (G);
\path (F) edge[bend right] (I);
\path (G) edge[bend left] (H);
\path (H) edge[bend left] (L);
\path (L) edge[bend left] (I);
\path (A) edge (F);
\path (C) edge (G);
\path (D) edge (H);
\path (B) edge (I);
\path (E) edge (L);
\end{tikzpicture}
\end{center}
\end{minipage}
\caption{The bipartite graph associated with the data matrix $E$ in
\eqref{eq:adjGPG} for $n=5$ and $k=1$ (left), which leads to the generalized Petersen graph $GPG(5,1)$
(right).}
\label{fig:GPG}
\end{figure}

We consider the data matrix represented by the graph in Figure~\ref{fig:GPG}
(left) for $n=5$ whose similarity matrix can be seen as the adjacency matrix of
the generalized Petersen graph $GPG(n,k)$ with a skip $k=1$; see
Figure~\ref{fig:GPG} (right). In this particular case, also the inner subgraph
is a cycle graph and the incidence matrix has the block structure
\begin{equation}\label{inc_GPG}
\widetilde{E} = \begin{bmatrix}
E^T & I_n & \bm{0}_{n} \\
\bm{0}_n & I_n & E^T
\end{bmatrix} \in \R^{2n \times 3n},
\end{equation}
where $E \in \R^{n \times n}$ is the incidence matrix of the cycle defined in~\eqref{eq:incidCycle}.
Its similarity matrix and Laplacian are, respectively
\begin{equation}\label{laplPG}
S = \begin{bmatrix}
F & I_n \\
I_n & F \end{bmatrix} \qquad
L = \begin{bmatrix}
\widetilde{F} & -I_n\\
-I_n & \widetilde{F}
\end{bmatrix}
\end{equation}
where $F$ and $\widetilde{F}$ are $n\times n$ circulant matrices given respectively by
\begin{equation}\label{matrixF}
F = \circu(3,1,\underbrace{0,\ldots,0}_{n-3},1)
\qquad 
\widetilde{F} = \circu(3,-1,\underbrace{0,\ldots,0}_{n-3},-1).
\end{equation}

\begin{theorem}\label{theo:FiedGPG}
Let $\widetilde{E}$ be the $2n\times 3n$ data matrix~\eqref{inc_GPG} of the
generalized Petersen graph $GPG(n,1)$. 
Then, the Fiedler value of the Laplacian matrix $L$ has multiplicity 2.
\end{theorem}
\begin{proof}
$L$ is a block circulant matrix with circulant blocks $\widetilde{F}$ and $-I_n$.
A block circulant matrix can be expressed as the sum of Kronecker products.
In our case, we have
$$
L = P_1 \otimes \widetilde{F} + P_2 \otimes (-I_n),
$$
where $P_1 = I_2$ and $P_2 = \circu(0,1)$.
More in general, one has $P_i = \circu(\bm{e}_i)$,
with $\bm{e}_i$ the $i$th canonical basis vector.

If we define the matrix-valued function 
$$
H(x) = x^0 \otimes \widetilde{F} + x^1 \otimes (-I_n),
$$
so that $H(P_2) = L$, it can be shown (see~\cite{kaveh2011}) that the spectrum
of $L$ is the union of the spectra of $H(\lambda_1)$ and $H(\lambda_2)$, being
$\lambda_1$ and $\lambda_2$ the eigenvalues of $P_2$.
Moreover, the eigenvectors of $L$ are given by the Kronecker products
$\bm{v}_i\otimes\bm{u}_j$, $i,j = 1,2$, where $v_i$ are the eigenvectors of
$P_2$ and $u_i$ are the eigenvectors of both $H(\lambda_1)$ and $H(\lambda_2)$. 

In our case, $\lambda_1 = 1$ and $\lambda_2 = -1$, so that
$H(\lambda_1) = \widetilde{F} - I_n$ and $H(\lambda_2) = \widetilde{F} + I_n$.
An immediate result is that the eigenvalues of $L$ are given by
$$
\mu_i = \begin{cases}
\sigma_i - 1 &  \quad \text{if} \quad i = 1, \dots, n \\
\sigma_{i-n} + 1 &  \quad \text{if} \quad i = n+1, \dots, 2n \\
\end{cases},
$$
where $\sigma_i$, $i = 1, \dots, n$, are the eigenvalues of the matrix $\widetilde{F}$. 
Since $\widetilde{F}$ is symmetric circulant, its eigenvalues are coupled (see
Theorem~\ref{theo:cycle}) and this completes the proof. 
\end{proof}

\smallskip
\begin{corollary}\label{corpet}
Let $\sigma$ be the second smallest eigenvalue of the matrix
$F$~\eqref{matrixF} and $\{\bm{w}_1, \bm{w}_2\}$ be a basis for the eigenspace
corresponding to $\sigma$. Then, $\sigma-1$ is the Fiedler value of the
Laplacian matrix $L$ given in~\eqref{laplPG} and $\{\bm{v}_1, \bm{v}_2\}$ is a
basis for the associated eigenspace, where
\begin{equation}\label{eig:GPG}
\bm{v}_1  = \begin{bmatrix} 1 \\ 1 \end{bmatrix} \otimes \bm{w}_1 
= \begin{bmatrix} \bm{w}_1 \\ \bm{w}_1 \end{bmatrix} 
\quad \text{and} \quad 
\bm{v}_2  = \begin{bmatrix} 1 \\ 1 \end{bmatrix} \otimes \bm{w}_2 
= \begin{bmatrix} \bm{w}_2 \\ \bm{w}_2 \end{bmatrix}.
\end{equation}
\end{corollary}

\begin{proof}
The proof follows from Theorem~\ref{theo:FiedGPG}, noting that $(1, 1)^T$ is
the eigenvector of $P_2$ associated to the eigenvalue $\lambda_1 = 1$.
\end{proof}

Since the eigenvectors of the matrix $\widetilde{F}$ are the columns of the normalized
Fourier matrix, we can obtain the set of admissible permutations from the
results obtained for the cycle graph. 
Indeed, the vectors $\bm{v}_1$ and $\bm{v}_2$ defined in~\eqref{eig:GPG} have
the same entries as the vectors $\bm{w}_1$ and $\bm{w}_2$ in~\eqref{eig_C}, but
each entry is doubled. This means that the components of a vector $\bm{x}$ in
the \emph{Fiedler plane} come in pairs.
Consequently, the number of the admissible permutations for a generalized
Petersen graph $GPG(n,1)$ is $2^n$ times the admissible permutations obtained
for a cycle graph.

For $n=4,5,6,7$, we expect at least 128, 480, 1920, and 6272 permutations,
respectively. Other admissible permutations may appear in case other
equalities occur between the entries of $\bm{v}_1$ and those of $\bm{v}_2$.
Since the graph has $2n$ nodes, formula \eqref{wrong} forecasts in this case
4320, 241920, $2.18\cdot 10^7$, and $2.87\cdot 10^9$ solutions, respectively.

\section{Two numerical methods to determine admissible permutations}\label{sec:methods}

A possible approach to find the admissible permutations associated to a Fiedler
vector in the presence of a multiple Fiedler value is to employ a randomized
algorithm. 

To this end, we developed a simple Monte Carlo approach.
In the case of a double Fiedler value, we considered $N$ random vectors in
$\R^2$ and used their components as coefficients of linear combinations of an
orthonormal basis for the corresponding eigenspace; see \eqref{Fiedvec}.
This procedure generates a set of random vectors belonging to a plane immersed
in $\R^n$, which can all be considered as legitimate ``Fiedler vectors''.
Each vector is then sorted and the corresponding permutations of indexes are
stored in the columns of a matrix. After removing all the repeated permutations
and the swapped ones, we obtain a set of allowed permutations of the $n$ nodes
in the considered graph.

The advantages of this approach are an easy implementation and its immediate
generalization to the case of a Fiedler value with multiplicity larger than 2.
The drawbacks are a large computational cost and the fact that this method
is not able to identify permutations deriving from specific values of the
coefficients of the linear combination; see for example the permutations
produced by the Fiedler vectors \eqref{sortvecs} for the modified star graph.
This aspects will be investigated in the numerical examples of
Section~\ref{sec:numexp}, where we will apply this numerical method and the
following one to the case studies considered in Section~\ref{sec:casestudies}.

In order to compute all the admissible permutations in the particular case of a
Fielder value with multiplicity 2, we developed a \emph{graphical method}
which is briefly described below.

Let the Laplacian matrix $L$ of a graph with $n$ nodes have a double Fiedler
value $\lambda_2$, and let
$$
\mathbf{v}=(v_1,v_2,\ldots,v_n)^T\quad \text{and} \quad
\mathbf{w}=(w_1,w_2,\ldots,w_n)^T
$$
be an orthogonal basis for the corresponding eigenspace $\mathcal{F}$ of
dimension 2.
The idea behind the method, described in Algorithm~\ref{alg:graphic}, is
considering the vector function
$$
\bm{f}(\gamma) = \bm{v}+\gamma \bm{w} 
= (f_1(\gamma),\ldots,f_n(\gamma))^T, 
$$
and represent its components $f_i(\gamma)=v_i+\gamma w_i$, $i=1,\dots,n$, as
straight lines in the Euclidean plane; see Figure~\ref{fig:cycle_lines}.

Computing the intersections of these lines (see line \ref{line9}) identifies
intervals in which the relative ordering of the components of $\bm{f}(\gamma)$
changes.
The position of the lines before the first intersection point (line
\ref{line20}) gives the reordering of the components of the linear combination
of $\mathbf{v}$ and $\mathbf{w}$ which corresponds to the first admissible
permutation of the nodes.
Then, new permutations are obtained by reordering the values of $f(\gamma)$ at
each intersection point and in the center point of each interval.
Indeed, an intersection point corresponds to a
swap of the components in the Fiedler vector, as $\gamma$ increases, and so to
a new permutation of the nodes.

The performances of the two procedures are analyzed and compared in the
numerical examples illustrated in the following section.

\begin{algorithm}[!ht]
\begin{algo}
\STATE \textbf{Requires:} Fiedler vectors $\bm{v},\bm{w}\in\R^n$ and
tolerance $\tau$
\STATE \textbf{Ensure:} matrix $P$ containing admissible node reorderings
\STATE $f(\gamma) = \bm{v}+\gamma\bm{w}$ 
\STATE $\Phi$ (2 columns matrix, initially empty, for intersections and their
	multiplicity)
\STATE $m=0$ (number of intersections found)
\FOR $i=1,\dots,n-1$
	\FOR $j=i+1,\dots,n$
		\IF {$|w_i-w_j|>\tau$}
			\STATE $\gamma_\text{int}=(v_i-v_j)/(w_j-w_i)$
				(new intersection abscissa) \label{line9}
			\STATE let $r\in\{1,\ldots,m\}$ such that 
				$|\gamma_\text{int}-\Phi_{r,1}|<\tau$,
				otherwise $r=0$
			\IF $r=0$ ($\gamma_\text{int}$ is not in $\Phi$)
				\STATE $m=m+1$, 
					$\Phi_{m,1}=\gamma_\text{int}$,
					$\Phi_{m,2}=1$ (add new intersection)
			\ELSE $\Phi_{r,2}=\Phi_{r,2}+1$ (increment multiplicity)
			\ENDIF
		\ENDIF
	\ENDFOR
\ENDFOR 
\STATE sort rows of $\Phi$ so that intersections are in increasing order
\STATE store in $P$ the permutations corresponding 
	to the possible orderings of $\bm{w}$
\STATE $\bm{y}_1 = f(\Phi_{1,1}-1)$
	(values of the lines in the first interval) \label{line20}
\STATE add to $P$ the permutations corresponding 
	to the possible orderings of $\bm{y}_1$
\FOR $i=1,\dots,m-1$
	\STATE $\bm{y}_1 = f(\Phi_{i,1})$ (left endpoint of $i$th interval)
		\label{line1}
	\STATE $\bm{y}_2 = f((\Phi_{i,1}+\Phi_{i+1,1})/2)$ 
		(center point of $i$th interval)
	\STATE add to $P$ the permutations corresponding to 
		the orderings of $\bm{y}_1$ and $\bm{y}_2$
\ENDFOR
\STATE $\bm{y}_1 = f(\Phi_{m,1})$ (last intersection) \label{line2}
\STATE $\bm{y}_2 = f(\Phi_{m,1}+1)$ (last interval)
\STATE add to $P$ the permutations corresponding to 
	the orderings of $\bm{y}_1$ and $\bm{y}_2$
\STATE remove from $P$ repeated or reversed permutations
\end{algo}
\caption{Graphic method for determining the admissible reorderings of the nodes
in a graph with a double Fiedler value}
\label{alg:graphic}
\end{algorithm}

\begin{figure}
\begin{center}
\includegraphics[width=.8\textwidth]{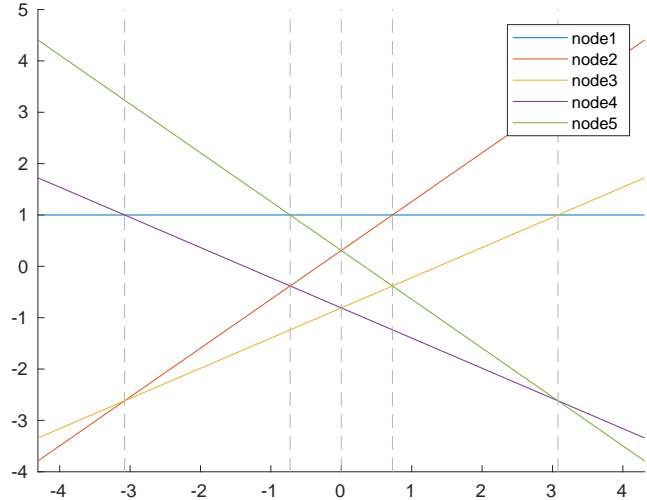}
\caption{Lines corresponding to the nodes in the cycle graph $C_n$ with $n=5$ nodes.}
\label{fig:cycle_lines}
\end{center}
\end{figure}

To illustrate the functioning of the graphical method, we consider
the cycle graph with $n=5$ nodes, depicted in Figure~\ref{fig:cycle}.
As pointed out in Section~\ref{subsec:cycle}, the admissible permutations
are $15$.
They can be obtained through the graphical method by
considering the swap of the indexes corresponding to the lines which intersect.
More precisely, in Figure~\ref{fig:cycle_lines} we report the lines
representing the functions $f_i(x)=v_i+x w_i$, for $i = 1, \ldots,5$, each one
corresponding to the node identified by the $i-$th component of the linear
combination of the Fiedler vectors $\mathbf{v}$ and $\mathbf{w}$. Due to
the fact that any vector in the eigenspace corresponding to the Fiedler values
can be expressed as in \eqref{eq:cycleFv}, there are intersection points with
the same abscissa highlighted by vertical dashed lines. 
As explained above, the first admissible permutation is obtained by considering
the position of the lines before the first intersection points and therefore it
is given by $(\ttt{5 4 1 3 2})$. The first vertical dashed line points out that
there are two pair of lines that intersect and consequently, the second set of
allowed permutations is obtained from the first one by considering the two
macro-nodes $(\ttt{1,4})$ and $(\ttt{2,3})$, that is, swapping the indexes
corresponding
to the lines that represent nodes 1 and 4 and nodes 2 and 3. Hence, the
additional permutations are given by
$$
(\ttt{5 1 4 3 2}), \qquad (\ttt{5 1 4 2 3}), \qquad (\ttt{5 4 1 2 3}).
$$ 
After the first intersection, the position of the lines gives the permutation
of the nodes $(\ttt{5 1 4 2 3})$, which has already been considered. The second
vertical dashed line, corresponding to the second intersection point, reveals
that two pair of lines intersect, i.e., we need to consider two macro-nodes,
namely $(\ttt{1,5})$ and $(\ttt{2,4})$. The new admissible permutations are
then
$$
(\ttt{5 1 2 4 3}), \qquad (\ttt{1 5 2 4 3}), \qquad (\ttt{1 5 4 2 3}).
$$ 
After the second intersection, the lines follow the order $(\ttt{1 5 2 4 3})$,
that is contained in the previous set. In correspondence to the third
intersection we have two pairs of lines which intersect, i.e., the Fielder
vectors have the two macro-nodes $(\ttt{2,5})$ and $(\ttt{3,4})$. In this case,
the encoded permutations are
$$
(\ttt{1 5 2 3 4}), \qquad (\ttt{1 2 5 4 3}), \qquad (\ttt{1 2 5 3 4}).
$$ 
After this intersection the permutation is $(\ttt{1 2 5 3 4})$, which has been
already taken into account. 
Considering the fourth vertical dashed line, which highlights that lines 
1-2 and 3-5 intersect, one obtains the admissible permutations 
$$
(\ttt{1 2 3 5 4}), \qquad (\ttt{2 1 5 3 4}), \qquad (\ttt{2 1 3 5 4}).
$$
After the fourth
intersection point, the position of the lines gives the permutation $(\ttt{2 1
3 5 4})$, already present in our set of permutations. 
The last intersection yields that lines 1-3 and 4-5 intersect, leading to
the further permutations 
$$
(\ttt{2 3 1 5 4}), \qquad (\ttt{2 1 3 4 5}), \qquad (\ttt{2 3 1 4 5}).
$$
The permutation $(\ttt{2 3 1 4 5})$, found in the last interval and coincident
with the last one of the previous set, coincides with the reverse of the first
one. Removing it leaves 15 admissible permutations of the indexes, which we
report as columns of the following matrix
$$
\begin{bmatrix}
5 & 5 & 5 & 5 & 5 & 1 & 1 & 1 & 1 & 1 & 1 & 2 & 2 & 2 & 2\\
4 & 1 & 1 & 4 & 1 & 5 & 5 & 5 & 2 & 2 & 2 & 1 & 1 & 3 & 1\\
1 & 4 & 4 & 1 & 2 & 2 & 4 & 2 & 5 & 5 & 3 & 5 & 3 & 1 & 3\\
3 & 3 & 2 & 2 & 4 & 4 & 2 & 3 & 4 & 3 & 5 & 3 & 5 & 5 & 4\\
2 & 2 & 3 & 3 & 3 & 3 & 3 & 4 & 3 & 4 & 4 & 4 & 4 & 4 & 5\\
\end{bmatrix}.
$$

\section{Numerical experiments}\label{sec:numexp}

In this section we report the results produced by the two methods introduced in
Section~\ref{sec:methods} for determining the admissible permutations of a set
of units, in the case the Fiedler value of the associated graph has
multiplicity 2.
To verify the performance of the methods, the graphical (see
Algorithm~\ref{alg:graphic}) and the Monte Carlo methods have been implemented
in Matlab R2021a and applied to the three case studies described in
Section~\ref{sec:casestudies}. The numerical experiments were performed on an
Intel Xeon Gold 6136 computer (16 cores, 32 threads) equipped with 128 GB RAM,
running the Linux operating system. 


The first computed example consists of finding the admissible permutations of
the nodes of a modified star graph $\widehat{\mathcal{S}}_n$ with data
matrix~\eqref{eq:mod_star}. As stated in Corollary~\ref{theo:star2bis}, the
Laplacian of the similarity matrix associated to the graph has a double Fiedler
value equal to 1. Since an orthogonal basis for the eigenspace $\mathcal{F}$
corresponding to the Fiedler value is known, every $\bm{x}\in\mathcal{F}$ can
be expressed by $\bm{x}=Q_2\bm{y}$, with $\bm{y}=[\alpha,\beta]^T$, as
in~\eqref{eigvecStar}.
As explained in detail in Section~\ref{sec:stargraph}, the permutations of
the nodes that yield a solution to the seriation problem are given by all the
possible reorderings of the entries of $\bm{x}$.

\begin{table}
\centering
\begin{tabular}{c|c|cc|cc}
& & \multicolumn{2}{c|}{Graphical method} & \multicolumn{2}{c}{Monte Carlo method}\\ 
n & $3(n-2)!$ & found perms & time & found perms & time \\ \hline
5 & 18 & 18 & 1.17e-01 & 14 & 1.27e-01 \\
6 & 72 & 72 & 1.57e-02 & 48 & 8.19e-02 \\
7 & 360 & 360 & 1.63e-02 & 216 & 2.58e-01 \\
8 & 2160 & 2160 & 9.02e-02 & 1200 & 5.22e+00 \\
9 & 15120 & 15120 & 1.11e+00 & 7920 & 8.80e+01 \\
10 & 120960 & 120960 & 2.21e+01 & 60480 & 1.76e+03 \\
\end{tabular}
\caption{Results obtained by applying the graphical and the Monte Carlo methods to the modified star graph with data matrix $E$~\eqref{eq:mod_star}.}\label{tab:modStar}
\end{table}

The results of the experiments concerning the application of the graphical and
the Monte Carlo methods to a graph $\widehat{\mathcal{S}}_n$ with a number of
nodes $n$ ranging from 5 to 10 are displayed in Table~\ref{tab:modStar}.
In particular, the second column contains the number $3(n-2)!$ of admissible
permutations for a modified star graph stated in~\eqref{right_star}.
It coincides with the number of admissible permutations found by the graphical
method, reported in the third column of the table.
We note that such number is one half of the estimate furnished by Equation
\eqref{wrong}, for $k=2$. For the following examples, the reduction with
respect to this estimate is even larger.

As the fifth column shows, the Monte Carlo method fails to
identify all the permutations, after considering $N=1000$ random linear
combinations of the orthonormal basis for the eigenspace $\mathcal{F}$.
We verified that increasing the value of $N$ up to 5000 the
performance of the method does not improve.
In this test, the graphical algorithm is, for every $n$, much faster than the
Monte Carlo method, as one can observe comparing the computing time in seconds
reported in the fourth and sixth columns of Table~\ref{tab:modStar}.

We remark that the failure of the Monte Carlo approach is due to the fact that
many admissible permutations result from specific values of the coefficients
$\alpha$ and $\beta$ in the linear combination~\eqref{eigvecStar}; see,
e.g.,~\eqref{sortvecs}. Assuming such values is an event with zero probability
in a random draw of real numbers, so it is very unlikely to occur in the
algorithm.
On the contrary, the graphical method explicitly considers equal components in
the Fiedler vectors when it processes intersections between the
lines; see lines \ref{line1} and \ref{line2} of Algorithm~\ref{alg:graphic}.


A similar comparison between the two methods has also been considered for the
cycle graph $\mathcal{C}_n$ analyzed in Section~\ref{subsec:cycle}.
The results are displayed in Table~\ref{tab:cycle}.
In this case, every vector $\bm{x}$ in the eigenspace associated with the
double Fiedler value of $\mathcal{C}_n$ can be represented as in
Equation~\eqref{eq:cycleFv}. 
In Section~\ref{subsec:cycle}, we have not been able to foresee the number of
admissible permutations for this graph, but the result we found for $n=4,5,6,7$
are confirmed by the outcome of the graphical method; see the second column in
Table~\ref{tab:cycle}. Again, the graphical method proves to be the fastest one
and the Monte Carlo method fails in recovering all the admissible permutations.
The reason for this failure is the same as discussed above.

\begin{table}
\centering
\begin{tabular}{c|cc|cc}
& \multicolumn{2}{|c|}{Graphical method} & \multicolumn{2}{|c}{Monte Carlo method}\\ 
n & found perms & time & found perms & time \\ \hline
4 & 8 & 1.53e-01 & 4 & 1.61e-01 \\
5 & 15 & 1.57e-01 & 7 & 4.87e-02 \\
6 & 30 & 1.48e-02 & 14 & 6.77e-02 \\
7 & 49 & 4.03e-03 & 13 & 6.13e-02 \\
8 & 88 & 4.90e-03 & 20 & 7.52e-02 \\
9 & 135 & 1.33e-02 & 23 & 7.68e-02 \\
10 & 230 & 5.25e-03 & 54 & 8.10e-02 \\
\end{tabular}
\caption{Results obtained by applying the graphical and the Monte Carlo methods
to the cycle graph with data matrix
$E$~\eqref{eq:incidCycle}.}\label{tab:cycle}
\end{table}


The results displayed in Table~\ref{tab:petersen} are obtained by applying the
two methods to the generalized Petersen graph $GPG(n,1)$. 
As discussed in Section~\ref{subs:petersen}, both the outer and the inner
subgraphs in $GPG(n,1)$ are cycle graphs and the total number of nodes is $2n$.
By following the discussion regarding the cycle graph and the results contained
in Theorem~\ref{theo:FiedGPG} and Corollary~\ref{corpet} it follows that
each vector $\bm{x}$ in the eigenspace corresponding to the Fiedler value has
$n$ macronodes of size two. Then, keeping into account the number of
permutations for a cycle, the admissible permutations of the nodes in
$GPG(n,1)$ are at least $2^n n$.

\begin{table}
\centering
\begin{tabular}{c|c|cc|cc}
& & \multicolumn{2}{|c|}{Graphical method} & \multicolumn{2}{|c}{Monte Carlo method}\\ 
n & $2^n n$ & found perms & time & found perms & time \\ \hline
5 & 160 & 5600 & 2.57e-01 & 160 & 1.61e+00 \\
6 & 384 & 48000 & 7.44e-01 & 384 & 1.38e+01 \\
7 & 896 & 192640 & 1.83e+01 & 896 & 3.99e+01 \\
8 & 2048 & 1546240 & 4.17e+02 & 2048 & 9.77e+01 \\
9 & 4608 & 5967360 & 3.10e+04 & 4608 & 2.38e+02 \\
\end{tabular}
\caption{Results obtained by applying the graphical and the Monte Carlo methods
to the Generalized Petersen graph with data matrix
$E$~\eqref{eq:adjGPG}.}\label{tab:petersen}
\end{table}

The second column of Table~\ref{tab:petersen} reports this minimum value for
the admissible permutations.
It is remarkable to observe that this is exactly the number of permutations
recovered by the Monte Carlo method.
Anyway, the real number of admissible permutations is much larger than that, as
testified by the results of the graphical method in the third column of the
table. This huge number of permutations requires a large computing time,
making the graphical method extremely slower than in the other examples.
Nevertheless, it is effective when computing the complete solution of the problem, while
the randomized approach it is not, even if in this case $N=5000$ random
Fiedler vectors have been used.

We analyzed the performance of both methods by means of the ``profiler''
available in Matlab. It turns out that the bottleneck for the execution time of
the algorithms are the tests for verifying that a new permutation does not
appear in the list of those already computed either in direct or reverse
ordering. When the number of admissible permutations is not too large, this
does not significantly affect the complexity of the graphical method, while it
does in the case of the generalized Petersen graph.

\section{Conclusions}\label{sec:last}
In this paper we studied the possible orderings of the Fiedler vector of a
graph, under the assumption that the Fiedler value has multiplicity larger than
one. The determination of such ordering is related to the solution of the
seriation problem.
We showed that, in the special case of a double Fiedler value, the number of
admissible permutations is smaller than the maximum number of permutations
allowed. In fact, it depends on the structure of the underlying bipartite
graph.
We examined three case studies for which it is possible to draw conclusions
about the solution of the problem, and we proposed a graphical method and a
randomized algorithm to list the admissible permutations.
Examples and numerical experiments illustrate the performance of the proposed
methods on the analyzed case studies.

\bibliographystyle{siam}	
\bibliography{bibliog}

\begin{thebibliography}{10}

\bibitem{atkins1998spectral}
{\sc J.~E. Atkins, E.~G. Boman, and B.~Hendrickson}, {\em A spectral algorithm
  for seriation and the consecutive ones problem}, SIAM J. Comput., 28 (1998),
  pp.~297--310.

\bibitem{booth1976testing}
{\sc K.~S. Booth and G.~S. Lueker}, {\em Testing for the consecutive ones
  property, interval graphs, and graph planarity using {PQ}-tree algorithms},
  J. Comput. Syst. Sci., 13 (1976), pp.~335--379.

\bibitem{brainerd1951place}
{\sc G.~W. Brainerd}, {\em The place of chronological ordering in
  archaeological analysis}, Am. Antiq., 16 (1951), pp.~301--313.

\bibitem{brusco06}
{\sc M.~J. Brusco and D.~Steinley}, {\em Clustering, seriation, and subset
  extraction of confusion data}, Psychol. Methods, 11 (2006), pp.~271--286.

\bibitem{chepoi1997}
{\sc V.~Chepoi and B.~Fichet}, {\em Recognition of {R}obinsonian
  dissimilarities}, J. Classif., 14 (1997), pp.~311--325.

\bibitem{pqser19}
{\sc A.~Concas, C.~Fenu, and G.~Rodriguez}, {\em {PQser}: a {M}atlab package
  for spectral seriation}, Numer. Algorithms, 80 (2019), pp.~879--902.

\bibitem{coxeter1950}
{\sc H.~S. Coxeter}, {\em Self-dual configurations and regular graphs}, Bull.
  Amer. Math. Soc., 56 (1950), pp.~413--455.

\bibitem{davis1979circulant}
{\sc P.~J. Davis}, {\em Circulant Matrices}, Wiley, New York, 1979.

\bibitem{genome}
{\sc M.~B. Eisen, P.~T. Spellman, P.~O. Brown, and D.~Botstein}, {\em Cluster
  analysis and display of genome-wide expression patterns}, P. Natl. Acad. Sci.
  U.S.A., 95 (1998), pp.~14863--14868.

\bibitem{estrada2010network}
{\sc E.~Estrada and D.~J. Higham}, {\em Network properties revealed through
  matrix functions}, SIAM Rev., 52 (2010), pp.~696--714.

\bibitem{fiedler1973algebraic}
{\sc M.~Fiedler}, {\em Algebraic connectivity of graphs}, Czech. Math. J., 23
  (1973), pp.~298--305.

\bibitem{fiedler1975property}
\leavevmode\vrule height 2pt depth -1.6pt width 23pt, {\em A property of
  eigenvectors of nonnegative symmetric matrices and its application to graph
  theory}, Czech. Math. J., 25 (1975), pp.~619--633.

\bibitem{fiedler1989laplacian}
\leavevmode\vrule height 2pt depth -1.6pt width 23pt, {\em Laplacian of graphs
  and algebraic connectivity}, Banach Center Publ., 25 (1989), pp.~57--70.

\bibitem{matharcheo}
{\sc F.~R. Hodson, D.~G. Kendall, and P.~Tautu}, {\em Mathematics in the
  Archaeological and Historical Sciences}, Edinburgh University Press,
  Edinburgh, 1971.

\bibitem{kaveh2011}
{\sc A.~Kaveh and H.~Rahami}, {\em Block circulant matrices and applications in
  free vibration analysis of cyclically repetitive structures}, Acta Mech., 217
  (2011), pp.~51--62.

\bibitem{laurent2017lex}
{\sc M.~Laurent and M.~Seminaroti}, {\em A {Lex-BFS-based} recognition
  algorithm for {R}obinsonian matrices}, Discret. Appl. Math., 222 (2017),
  pp.~151--165.

\bibitem{laurent2017}
\leavevmode\vrule height 2pt depth -1.6pt width 23pt, {\em Similarity-{F}irst
  {S}earch: a new algorithm with application to {R}obinsonian matrix
  recognition}, SIAM Discret. Math., 31 (2017), pp.~1765--1800.

\bibitem{mirkin1984}
{\sc B.~G. Mirkin and S.~N. Rodin}, {\em Graphs and Genes}, vol.~11 of
  Biomathematics, Springer-Verlag, Berlin, 1984.

\bibitem{ortega1960}
{\sc J.~M. Ortega}, {\em On {S}turm sequences for tridiagonal matrices}, J.
  ACM, 7 (1960), pp.~260--263.

\bibitem{petrie}
{\sc W.~M.~F. Petrie}, {\em Sequences in prehistoric remains}, J. R. Anthropol.
  Inst., 29 (1899), pp.~295--301.

\bibitem{ps05}
{\sc P.~Piana~Agostinetti and M.~Sommacal}, {\em Il problema della seriazione
  in archeologia}, Rivista di Scienze Preistoriche, LV (2005), pp.~29--69.

\bibitem{prea2014}
{\sc P.~Pr{\'e}a and D.~Fortin}, {\em An optimal algorithm to recognize
  {R}obinsonian dissimilarities}, J. Classif., 31 (2014), p.~351.

\bibitem{robinson1951method}
{\sc W.~S. Robinson}, {\em A method for chronologically ordering archaeological
  deposits}, Am. Antiq., 16 (1951), pp.~293--301.

\bibitem{seston2008}
{\sc M.~Seston}, {\em Dissimilarit{\'e}s de {R}obinson: algorithmes de
  reconnaissance et d'approximation}, PhD thesis, Aix Marseille 2, 2008.

\bibitem{watkins1969}
{\sc M.~E. Watkins}, {\em A theorem on tait colorings with an application to
  the generalized {P}etersen graphs}, J. Comb. Theory, 6 (1969), pp.~152--164.

\bibitem{wilkinson}
{\sc J.~H. Wilkinson}, {\em The Algebraic Eigenvalue Problem}, vol.~87,
  Clarendon Press, Oxford, 1965.

\end{thebibliography}

\end{document}